\documentclass[12pt]{article}
\usepackage{amsmath,amssymb,amsfonts,amsthm}
\usepackage[all]{xy}

\newcounter{item}[section]
\newcounter{kirshr}
\newcounter{kirsha}
\newcounter{kirshb}
\newenvironment{enumroman}{\setcounter{kirshr}{1}
\begin{list}{(\roman{kirshr})}{\usecounter{kirshr}} }{\end{list}}
\newenvironment{enumarab}{\setcounter{kirshb}{1}
\begin{list}{(\arabic{kirshb})}{\usecounter{kirshb}} }{\end{list}}
\newtheorem{theorem}{Theorem}[section]

\newtheorem{lemma}[theorem]{Lemma}
\newtheorem{corollary}[theorem]{Corollary}
\newenvironment{demo}[1]{\noindent{\bf #1.}\upshape\mdseries}
{\nopagebreak{\hfill\rule{2mm}{2mm}\nopagebreak}\par\normalfont}
\theoremstyle{definition}

\newtheorem{example}[theorem]{Example}
\newtheorem{definition}[theorem]{Definition}

\def\C{{\mathfrak{C}}}
\def\Fm{{\mathfrak{Fm}}}

\def\At{{\sf At}}

\def\Nr{{\mathfrak{Nr}}}
\def\Fr{{\mathfrak{Fr}}}
\def\Sg{{\mathfrak{Sg}}}
\def\Fm{{\mathfrak{Fm}}}
\def\A{{\mathfrak{A}}}
\def\B{{\mathfrak{B}}}
\def\C{{\mathfrak{C}}}
\def\D{{\mathfrak{D}}}
\def\M{{\mathfrak{M}}}

\def\CA{{\bf CA}}

\def\Dc{{\bf Dc}}

\def\K{{\bf K}}
\def\K{{\bf K}}

\def\Rd{{\mathfrak{Rd}}}
\def\(R)RA{{\bf (R)RA}}

\def\Dc{{\bf Dc}}
\def\R{\mathbb{R}}
\def\Dc{{\bf Dc}}

\def\Sc{{\bf Sc}}

 \def\CA{{\sf CA}}
\def\B{{\sf B}}
\def\G{{\sf G}}
\def\w{{\sf w}}
\def\y{{\sf y}}
\def\g{{\sf g}}

\def\r{{\sf r}}
\def\K{{\sf K}}
 \def\Cm{{\mathfrak{Cm}}}
\def\Nr{{\mathfrak{Nr}}}
\def\SNr{{\bf S}{\mathfrak{Nr}}}

\def\cyl#1{{\sf c}_{#1}}

\def\sub#1#2{{\sf s}^{#1}_{#2}}

\def\Ra{{\mathfrak{Ra}}}
\def\Ca{{\mathfrak{Ca}}}
\def\set#1{\{#1\} }
\def\Ra{{\mathfrak{Ra}}}
\def\Nr{{\mathfrak{Nr}}}
\def\Tm{{\mathfrak{Tm}}}
\def\A{{\mathfrak{A}}}
\def\B{{\mathfrak{B}}}
\def\C{{\mathfrak{C}}}
\def\D{{\mathfrak{D}}}

\def\CA{{\bf CA}}

\def\G{{\bf G}}

\def\L{{\mathfrak{L}}}
\def\R{\cal{R}}

\def\ws{winning strategy}

\def \set#1{\{#1\} }

\def\Nr{{\mathfrak{Nr}}}
\def\Fr{{\mathfrak{Fr}}}
\def\Sg{{\mathfrak{Sg}}}
\def\Fm{{\mathfrak{Fm}}}
\def\Rd{{\mathfrak{Rd}}}
\def\Ig{{\mathfrak{Ig}}}
\def\CA{{\bf CA}}

\def\K{{\bf K}}
\def\L{{\bf L}}

\def\(R)RA{{\bf (R)RA}}

\def\Dc{{\bf Dc}}
\def\R{\mathbb{R}}
\def\Dc{{\bf Dc}}

\def\Nr{{\mathfrak{Nr}}}
\def\Fr{{\mathfrak{Fr}}}
\def\Sg{{\mathfrak{Sg}}}
\def\Fm{{\mathfrak{Fm}}}
\def\Rd{{\mathfrak{Rd}}}
\def\Ig{{\mathfrak{Ig}}}
\def\CA{{\bf CA}}

\def\K{{\bf K}}
\def\L{{\bf L}}

\def\(R)RA{{\bf (R)RA}}

\def\Dc{{\bf Dc}}
\def\R{\mathbb{R}}
\def\Dc{{\bf Dc}}

\def\R{\mathbb{R}}

\def\R{\mathbb{R}}
\def\Dc{{\bf Dc}}

\def\Sc{{\bf Sc}}

 \def\CA{{\sf CA}}

\def\M{{\mathfrak{M}}}

\def\At{{\mathfrak{At}}}

\def\G{{\mathfrak{G}}}

\def\K{{\bf K}}

\def\tp{{\sf tp}}

\def\sub#1#2{{\sf s}^{#1}_{#2}}
\def\cyl#1{{\sf c}_{#1}}

\def\pa{$\forall$}
\def\pe{$\exists$}

\def\ef{Ehren\-feucht--Fra\"\i ss\'e}

\def\nodes{{\sf nodes}}

\def\A{{\mathfrak{A}}}
\def\B{{\mathfrak{B}}}
\def\C{{\mathfrak{C}}}
\def\D{{\mathfrak{D}}}

\def\Fm{{\mathfrak{Fm}}}
\def\Ra{{\mathfrak{Ra}}}
\def\Nr{{\mathfrak{Nr}}}
\def\F{{\mathfrak{F}}}
\def\CA{{\bf CA}}

\def\set#1{ \{#1\}}

\def\Ca{{\mathfrak Ca}}

\def\pe{$\exists$}
\def\pa{$\forall$}
\def\Cm{{\mathfrak Cm}}
\def\Sg{{\mathfrak Sg}}

\def\ls { L\"owenheim--Skolem}
\def\At{{\sf At}}
\def\Id{{\sf Id}}

\def\rng{{\sf rng}}
\def\dom{{\sf dom}}
\def\Fl{{\mathfrak{Fl}}}

\def\w{{\sf w}}
\def\g{{\sf g}}
\def\y{{\sf y}}
\def\r{{\sf r}}
\def\Co{{\sf Co}}
\def\cyl#1{{\sf c}_{#1}}
\def\sub#1#2{{\sf s}^{#1}_{#2}}

\def\d{Dedekind-MacNeille}

\def\swap#1#2{{\sf s}_{[#1, #2]}}

\def\ws{winning strategy}
\def\ef{Ehren\-feucht--Fra\"\i ss\'e}

\def\Rl{\mathfrak{Rl}}

\def\y{{\sf y}}
\def\g{{\sf g}}
\def\r{{\sf r}}
\def\w{{\sf w}}
\def\MA{{\sf MA}}

\def\Dc{{\sf Dc}}
\def\Sc{{\sf Sc}}
\def\Ss{{\sf Ss}}

\def\PMA{{\sf PMA}}

\def\TCA{{\sf TCA}}
\def\CA{{\sf CA}}
\def\TDc{{\sf TDc}}
\def\R{{\sf R}}
\def\TPCA{{\sf TPCA}}
\def\L{{\mathfrak{L}}}
\def\ls { L\"owenheim--Skolem}

\title{Algebraisable versions of topological predicate logic, Part 1\\
{\it Toplogical logic via cylindric algebras}}
\author{Tarek Sayed Ahmed}

\begin{document}
\maketitle
\begin{abstract}
\noindent  Motivated by questions like:  which spatial structures may be characterized by means of modal logic, 
what is the logic of space, how to encode
in modal logic different geometric relations, topological logic provides a framework for studying the confluence of the topological semantics for 
$\sf S4$ modalities, based on topological spaces 
rather than Kripke frames, with the $\sf S4$ 
modality induced by the interior operator. 

Following research initiated by Sgro, and further pursued algebraically by Georgescu, 
we prove an interpolation theorem and an omitting types theorem 
for various extensions of predicate topological logic and Chang's modal logic. Our proof is algebraic addressing
expansions of cylindric algebras using interior operators and  boxes, respectively. 
Then we proceed like is done in abstract algebraic
logic by studing algebraisable extensions of both logics; obtaining a plethora of results
on the amalgamation property for various subclasses of their algebraic counterparts, which are varieties.
As a sample, we show that the free algebras of infinite dimensions 
enjoy several weak forms of interpolation, a property equivalent to the fact
that the class of simple algebras have the amalgamation property, 
but they fail the usual Craig interpolation  property, because the whole variety fails to have
the amalgamation property. Such interpolation properties fail for finite dimensions $>1$.

Notions like atom-canonicity and complete representations are approached for finite dimensional topological 
cylindric algebras. The logical consequences of our algebraic results are carefully worked out for infinitary extensions of Chang's 
predicate modal logic and finite versions thereof, by restricting to $n$ variables, $n$ finite, 
viewed as a propositional multi-dimensional modal logic, and $n$ products of bimodal whose
frames are of  the form $(U, U\times U, R)$ where $R$ is a pre-order, endowed with diagonal 
constants.  We show that for any finite $n>2$ such modal logics, though canonical hence Kripke complete,  
are necessarily non-finitely axiomatizable,
furthermore, any axiomatization must contain infinitely many propositional variables, 
infinitely many diagonal constants, and infinitely many non-canonical sentences; hence they are only barely canonical. 
In particular, they are 
not Sahlqvist axiomatizable; and even more they cannot be axiomatized by modal formulas with first order corespondances on 
their Kripke frames. For $n\leq 2$, such logics are are finitely axiomatizable by Sahlqvist modal  formulas, 
they are decidable for $n=1$ and undecidable for $n=2$. We shall also deal with guarded versions of such topological 
$n$ modal logics (by relativizing states to guards) 
proving that they have the finite model property, are decidable, and 
finitely axiomatizable, for any finite $n$.
 
The paper has four parts, this is the first.
\footnote{Topological logic, Chang modal logic, cylindric algebras, representation theory, amalgamation, congruence extension, interpolation
Mathematics subject classification: 03B50, 03B52, 03G15.}
\end{abstract}

\section{Introduction}

\subsection{Universal logic}

Universal logic is the field of logic that is concerned with giving an account of what features are common to all logical structures. 
If slogans are to be taken seriously, then universal logic is to logic what universal algebra is to algebra.
The term ``universal logic' was introduced in the 1990s 
by Swiss Logician Jean Yves Beziau but the field has arguably existed for many decades. 
Some of the works of Alfred Tarski in the early twentieth century, on metamathematics and in algebraic logic, 
for example, can be regarded undoubtedly, in retrospect, as fundamental contributions to universal logic. 
Indeed, there is a  whole well established 
branch of algebraic logic,
that attempts to deal with the universal notion of a logic. Pioneers in this branch include Andr\'eka
and    N\'emeti \cite{AN75} and Blok and Pigozzi \cite{BP}. 
The approach of Andr\'eka and N\'emeti though is more general, since, unlike the approach in \cite{BP} which 
is purely syntactical,  it allows semantical notions stimulated via 
so-called `meaning functions' \cite{ANS}, to be defined below. Another universal approach to many cylindric-like algebras was implemented in 
\cite{univl} in the context of the very general notion of what is known
in the literature as systems of varieties definable by a Monk's schema \cite{ST, HMT2}.

One aim of universal logic is to determine the domain of validity of such and such metatheorem 
(e.g. the completeness theorem, the Craig interpolation 
theorem, or the Orey-Henkin omitting types theorem of first order logic) and to 
give general formulations of metatheorems in broader, or even entirely other contexts.
This is also done in algebraic logic, by dealing with modifications and variants of first order logic resulting in a natural way
during the process of {\it algebraisation}, witness for example the omitting types theorem proved in  \cite{Sayed}. 

This kind of investigation  is extremely potent  for applications and helps to make the distinction 
between what is really essential to a particular logic and what is not. 

During the 20th 
century, numerous logics have been created, to mention only a few: intuitionistic logic, modal logic, topological logic, spatial logic, 
dynamic logic,
many-valued logic, fuzzy logic, relevant logic, 
para-consistent logic, non monotonic logic, 
etc. Universal logic is not a new addition (not a new logic), it is rather a way of unifying this multiplicity of logics 
by developing general means and concepts that can encompass all hitherto existing logics allowing
a uniform treatment of their meta theories, so in this respect it resembles category theory 
whose main concern is to highlight adjoint situations in various branches
of mathematics.

Universal logic also helps to clarify {\it basic concepts} explaining what is an extension and what is a deviation of a given logic, 
what does it mean for a logic to be equivalent, stronger, or interpretable into another one. 
It allows to give precise definitions of notions often discussed by philosophers 
like {\it truth-functionality, extensionality, logical form}, etc.
But such issues are at the heart of research in algebraic logic as well.

\subsection{Algebraic logic}

Traditionally, algebraic logic has focused on
the algebraic investigation of particular classes of algebras, the most famous
are Tarski's cylindric algebras and Halmos' polyadic algebras, whether or not they could be connected to some known
assertional system by means of the Lindenbaum-Tarski method of forming algebras of formulas. Viewing the set of formulas as an algebra with
operations induced by the logical connectives, logical equivalence
is a congruence relation on the formula algebra.

However, when such a connection could be established, there was
interest in investigating the relationship between various
metalogical properties of the logistic system and the algebraic
properties of the associated class of algebras (obtaining what are
sometimes called``bridge theorems"); so in a way algebraic logic can be viewed as the natural interface between logic 
(in a broad sense) and universal algebra. 

For example, it was discovered
at quite an early stage of the development of the subject that there is a natural relation between the interpolation theorems
of intuitionistic, intermediate propositional calculi,
and the amalgamation properties of varieties of Heyting algebras, due to several authors, including Tarski, Jonsson, 
Rasiowa, Sikorski and others. 
Similar connections were investigated between interpolation theorems
in the classical predicate calculi and congruence extension properties and amalgamation results in varieties of
cylindric and polyadic algebras; pioneers in this connection include Comer, Johnson, 
Diagneault and Pigozzi \cite{AUamal, P, MS, MStwo, AUU}.

Interpolation theorems require the presence of at least a partial order, but the 
congruence extension property
for an algebra, and for that matter the amalgamation property for a class of algebras are more universal notions, and lend themselves to
wider contexts.

Qouting Pigozzi  from \cite{P} `It is always exciting for a mathematician when close connections
are discovered between seemingly distant notions 
and results from two different branches of mathematics, 
and this is especially true when the 
notions and results involved are important ones and the 
focus of considerable research in their respective areas.
Thus for instance, some recent developments have brought to light
close and unexpected connections between two groups of results - 
metalogical interpolation theorems
of which 
the first and best known is Craig's interpolation theorem for first order logic 
and the algebraic theorems to the effect that certain classes of algebras have the 
amalgamation property."

The framework of the work of Pigozzi in \cite{P} was cylindric algebras an equational formalism of first order logic.
On the other hand, 
Georgescu has shown that the strongly related representation theory of Halmos' polyadic algebras 
can be applied to prove a completeness theorem  
for  many predicate logics, like tense logic, $\sf S4$ modal logic, intuitionistic logic, Chang $\sf S5$ modal logic and 
topological logic \cite{g, g2, g3, g4, g5}. 

In this paper, among many other things, we carry out an analogous investigation but for {\it algebraisable extensions and 
/or versions and modifications} of 
predicate topological logic and Chang modal logic, using
the well developed machinery of the theory of cylindric algebras, and `bridge theorems' in abstract algebraic 
logic.

\subsection{Topological logic and Chang's modal logic}

Topological logic was introduced by Makowsky and Ziegler \cite{z, z2}, and Sgro \cite{s}.
Such logics have a classical semantics with a topological flavour, addessing spatial logics
and their study was approached  using algebraic logic by Georgescu \cite{g}, the task that we further pursue in this 
paper. Topological logics are apt for dealing with {\it logic} and {\it space}; 
the overall point is to take a common mathematical model of space (like a topological space)
and then to fashon logical tools to work with it.

One of the things which blatantly strikes one when studing elementary topology is that notions like open, closed, dense 
are intuitively very transparent, and their
basic properties are absolutely straightforward to prove. However, topology
uses second order notions as it reasons with sets and subsets of `points'. This might suggest that like
second order logic, topology ought to be computationally very complex.
This apparent dichotomy between the two paradigms
vanishes when one realizes that a large portion of
topology can be formulated as a  simple modal logic, 
namely, $\sf S4$! 
This is for sure an asset 
for modal logics tend to be much easier to handle than first order logic
let alone second order.

The project of relating topology to modal logic begins
with work of Alfred Tarski and J.C.C McKinsey \cite{tm}. Strictly speaking Tarski and McKinsey did not work with
modal logic, but rather with its algebraic counterpart, namely, Boolean algebras 
with operators which is the approach we adopt here; the operators they studied where the closure operator induced on what they called
{\it the algebra of topology}, certainly a very ambitious title, giving the impression
that the paper aspired to completely {\it algebraise topology}. 

In retrospect McKinsey and Tarski showed, that the Stone representation theorem for Boolean algebras extend to algebras 
with operators to give topological semantics for classical propositional modal logic, 
in which the `neccessity' operation is modelled by taking the interior (dual operation) 
of an arbitrary subset of
topological space. Although the topological completeness of $\sf S4$ has been well known for quite 
a long time, it was until recently considered as some
exotic curioisity, but certainly having mathematical value. 
It was in the 1990-ies that the work of McKinsey and Tarski, came to the front scene of modal logic (particularly spatial modal logic),
drawing serious attention of many researchers and inspiring
a lot of work stimulated basically by questions concerning the `modal logic of space'; 
how to encode in modal logic different geometric relations? 
A point of contact here between topological spaces, geometry, and cylindric algebra theory
is the notion of dimension. 

From the modern point of view one introduces a basic {\it modal language} with a set $\At$ of atomic propositions, 
the logical Boolean connectives $\land$, $\neg$ and a modality $I$ 
to be interpreted as the 
{\it interior operation.}
Let $X$ be a topological space. The modal language $\L_0$ is interpreted on such a space $X$ together with an {\it interpretation map}
$i:\At \to \wp(X)$. For atomic $p\in \At$, $i(p)$ says which points satisfy $p$. We do not require that $i(p)$ is open.
$(X, i)$ is said to be a {\it topological model}. Then $i$ extends to all $\L_0$ formulas by interpreting negation as 
complement relative
to $X$, conjunction as intersection and $I$ as the interior operator.
In symbols we have:
\begin{align*}
i(\neg \phi)&=X\sim i(\phi),\\
i(\phi\land \psi)&=i(\phi)\cap i(\psi),\\
i(I(\phi))&={\sf int}i(\phi).
\end{align*}
The main idea here is that the basic properties of the
Boolean operations
on sets as well as the salient
topological operations like interior and its dual the closure, correspond to to schemes of sentences.
For example, the fact that the interior operator is idempotent
is expressed by $$i((II\phi)\leftrightarrow (I\phi))=X.$$
The natural question to ask about this language and its semantics is:
Can we characterize in an enlightning way the sentences $\phi$ with the property that for 
all topological models $(X, i)$, $i(\phi)=X$; these are
the topologically valid sentences. 
They are true at all points in all spaces under whatever  interpretation.
More succintly, do we have a nice {\it completeness} theorem?

Tarski and McKinsey proved that the topologically valid sentences are exactly those provable in the modal logic 
$\sf S4$. 
$\sf S4$ has a seemingly different semantics using standard {\it Kripke frames.}
Now $X$ is viewed as the set of {\it possible worlds}. In $\sf S4$, $I$ is read as {\it all points which the current point relates to}.
To get a sound interpretation 
of $\sf S4$ we should require that the current
point is {\it related to itself}. Therefore  we are led to the notion of a pre-ordered model. A 
{\it pre-ordered model}  is defined to be a triple $(X, \leq, i)$ where $(X, \leq)$ is a pre-order and $i: \At\to \wp(X)$
where
$$i(I(\phi))=\{x: \{y:x\leq y\}\subseteq i(\phi)\}.$$
Temporally world $x'\in X$ is a {\it successor} of world $x\in X$ if $x\leq x'$,  $x$ and $x'$ are equivalent worlds if further 
$x\leq x'$ and $x'\leq x$.
We have a completely analogous result here; $\phi$ is valid in pre-ordered models if 
$\phi$ is provable in $\sf S4$.

One can prove the equivalence of the two systems using only topologies on finite sets.
Let $(X, \leq)$ be a pre-order. Consider the {\it Alexandrov topology} on $X$, the open sets are the sets closed upwards in the order.
This gives a topology, call it $O_{\leq}.$ A correspondence 
between topological models and pre-ordered models can thereby be obtained,  
and as it happens we have
for any pre-ordered model $(X, \leq, i)$, all $x\in X$, and all $\phi\in \L_0$
$$x\models \phi \text { in } (X,\leq, i) \Longleftrightarrow x\models \phi \text { in } (X, O_{\leq}, i).$$

Using this result together with the fact that sentences satisfiable in $\sf S4$ have finite topological models, thus 
they are Alexandov topologies, one can show that the semantics 
of both systems each is interpretable in the other; they are equivalent.
We can summarize the above discussion in the following neat theorem, that we can and will 
attribute to McKinsey, Tarski
and  Kripke; this historically is not very accurate. For a topological space $X$ and $\phi$ an $\sf S4$ 
formula we write $X\models \phi$, if $\phi$ is valid topologically in 
$X$ (in either of the senses above). For example, $w\models \Box \phi$ iff 
for all  $w'$ if $w\leq w'$, then $w'\models \phi$, where $\leq$ is the relation $x\leq y$ 
iff $y\in {\sf cl}\{x\}$.

\begin{theorem} (McKinsey-Tarski-Kripke) 
Suppose that $X$ is a dense in itself metric space (every point is a limit point) 
and $\phi$ is a modal $\sf S4$ formula. Then the following are equivalent
\begin{enumarab}
\item $\phi\in \sf S4.$
\item$\models \phi.$
\item $X\models \phi.$
\item $\mathbb{R}\models \phi.$
\item $Y\models \phi$ for every finite topological space $Y.$
\item $Y\models \phi$ for every Alexandrov space $Y.$
\end{enumarab}
\end{theorem}

One can say that finite topological space or their natural extension to Alexandrov topological spaces reflect faithfully 
the $\sf S4$ semantics, and that arbitrary topological spaces generalize $\sf S4$ frames. On the other hand,
every topological space gives rise to a normal modal logic. Indeed $\sf S4$ is the modal logic of $\mathbb{R}$, or 
any metric that is  separable and dense in itself 
space, or all topological spaces, as indicated above. Also a 
recent result is that it is also the modal logic of the Cantor set, which is known to be Baire isomorphic to 
$\mathbb{R}$. 

But, on the other hand,  modal logic is too weak to detect interesting properties
of $\mathbb{R}$, for example it cannot distinguish between $[0, 1]$ and $\mathbb{R}$ despite
their topological disimilarities, the most striking one being compactness; $[0, 1]$ is compact, but $\mathbb{R}$ is 
not. 

To make $\sf S4$ stronger and more expressive, one  can enrich the modal language. {\it Hybrid languages} are such; 
they have proposition letters called 
{\it nominals} and {\it global modality}. Nominals 
denote singleton sets and 
global modality allows to say that a formula holds somewhere. In Hybrid 
modal logic one can say that the closure of any singleton is itself, by $\Diamond i\to i$ ($i$ a nominal) 
which is valid in 
$\mathbb{R}$ but not in spaces that are not $T_1$. 
Hence $T_1$ is definable by nominalis and the Hybrid logic of $\mathbb{R}$ 
is {\it not} that of any topological space and so it is stronger than $\sf S4$.

One can also enrich the language of $\sf S4$ with a modal operator $[a]$ giving it a temporal dimension; $[a]$ interpreted as 'next'. 
If $X$ is a topological space and $f:X\to X$ is a continous function, then the pair $(X,f)$ is called a 
{\it dynamic space  over $X$}. If $f$ is the identity function, then this is a {\it static} space; it is nothing more than $\sf S4$, 
because the 'next' world 
is only the same world. The field of {\it dynamic topological logic} dealing with dynamic spaces over topological spaces, 
modalizing dynamical systems,
is quite an active field of research ; providing a unifying framework
for studying the confluence of three rich research areas: the topological semantics for $\sf S4$, 
topological dynamics, and temporal logic.

\begin{definition} A {\it dynamic topological model on $X$} consists of a dynamic space $(X, f)$ over $X$ and a valuation $i$ of propositional 
variables to subsets of $X$  such that 
\begin{align*}
i(\neg \phi)&=X\sim i(\phi),\\
i(\phi\land \psi)&=i(\phi)\cap i(\psi),\\
i(I(\phi))&={\sf int}i(\phi)\\
i([a]\phi)&=f^{-1}(i(a)).
\end{align*}
\end{definition}
The resulting modal logic is called $\sf S4C$. 
We have a completeness theorem here as well:
\begin{theorem} For any formula $\phi$ the following are equivalent:
\begin{enumarab}
\item $\sf S4C\vdash \phi$
\item $\phi$ is topologically valid
\item $\phi$ is true in any finite topological space
\end{enumarab}
\end{theorem}
But here there are derivable formulas that are not valid in $\mathbb{R}$. 
However,  such dynamic topological logics, have a very interesting completeness theorem, namely, 
that for any formula that is {\it not} derivable, there exists a countermodel in $\mathbb{R}^n$ 
for $n$ sufficiently large, where the upper bound
of the dimension, namely $n$,  is charaterized by the modal depth of
such a formula. The techniques used suggest that such a modality, or perhaps a similar one, 
may be used characterize the geometric notion of dimension, but further research is 
needed.

Topological interpretations of propositional topological logic 
were recently extended in a natural way to arbitrary theories of full first order logic by 
Awodey and Kishida using so-called {\it topological sheaves} to
interpret domains of quantification \cite{ak}. 

They prove that $\sf S4\forall$ (predicate $\sf S4$ logic) is  complete with respect to such extended topological semantics, using
techniques related to recent work in topos theory. Indeed, historically Sheaf semantics was 
first introduced by topoi theorists for higher order intuitionistic logic, and has beebn applied
to first order modal logic, by both modal and categorical logicians. 

Sheaves or pre-sheaves 
taken over a possible world structure- most notably Kripke sheaves
over a Kripke frame can be regarded as extending
the structure to the first order level with variable domains 
of individuals; the modality arises naturally from a gemoetric 
morphism between the topos of such  sheaves of the associated world structures.
The completeness proof in essence 
is a translation of a Henkin construction; implementing a so-called `de modalization process' 

Given a first order modal language, the construction gives a first order non-modal language and a surjective interpretation from the former to
the latter, along this interpretation we can have a non-modal version of a given modal
theory.
So the modal predicate language is reduced to an ordinary predicate one by elimninating the $\sf S4$ operation, 
but its models of consistent theories built by a Henkin usual construction are interpretable, or rather
give rise to,  models of  the original predicate modal language.

In this paper we also study algebraically a predicate version  
of the modal topological logic described above. One way of doing this is that we deal with the same syntax in 
\cite{ak}, but we alter the semantics, dealing with usual Kripke semantics, proving a stronger result, namely,
an interpolation theorem; we also touch on dynamic topological predicate logics.
  
But next we proceed differently; the modus operandi, and the overall goals of our work are 
different, too. 

We assume that the models carry a topology, but now models are more complex; 
they are structures
for first order logic.
Consider such a structure $\M$ for a given first order language in a certain signature having a sequence of variables of order type $\omega$  
and assume that its underlying set $M$ is endowed with a topology. Then the set of all assignments
satisfying a formula $\phi$ interpreted the usual Tarskian way 
can be seen as an $\omega$-ary relation on $M$, call 
it $\phi^{\M}$.

Unlike the approach adopted in \cite{ak}, where there is only one $\sf S4$ modality, 
here for each $k<\omega$, we add to the syntax an operation $I_k$ interpreted at a formula $\phi$ 
as those sequences $s$ satisfying $\phi$
except that at the $k$th 
co-ordinate we require that $s(k)$ is in the interior of 
the $kth$ component of $\phi^{\M}$, so we get a smaller set than $\phi^{\M}$.
So here we have $\omega$ many modalities, not just one, each acting on 
one component of the set of sequences satsifying a given formula; when we deal with only finitely many variables $m$,
we will have $m$ modalities, but we shall also look at cylindrifiers as diamonds dealing with $2m$ 
multi dimensional propositional 
modal logic.

A {\it completeness and an omitting types theorem} are proved algebraically 
by Georgescu in \cite{g} for usual first order logic (with infinitely may variables) 
with such semantics  involving the interior operators induced by a topology on the base of models.

But as it happens, there is also a {\it modal approach} to topological predicate logic \cite{Chang, g, z2}. 
Each interior operator can  again be viewed as a {\it modality} $\Box_i$, called {\it Chang's modal operator}
and its semantics is specified by a {\it Chang system} for a model $\M$, which is a function
$V:M\to \wp(\wp (M))$. The semantics is now defined as follows: 
$$s\in \Box_i\phi^{\M}\Longleftrightarrow \{u\in M: s^i_u\in \phi^{\M}\}\in V(s_i).$$
Here $s^i_u$ is the function that is like $s$ everywhere except that its value at $i$ is $u$.
If $M$ carries a topology $\mathfrak{O}$ say, then this gives a natural Chang system
defined by $V(x)=\mathfrak{O}$, for all $x\in M$. A completeness theorem was also 
proved by Georgescu \cite{g2} for Chang's modal logics using polyadic algebras.
So one can view topological modals as  nice semantics for Chang's $\sf S4$ modal logic.

\subsection{The process of algebraisation}

We go further in the analysis carried out in \cite{g, g2} using also an algebraic approach, but on a wider scale, 
proving  stronger and much more results. 
In the above cited references Georgescu uses the representation theory of {\it locally finite polyadic algebras with equality}, here we use
the representation theory of {\it dimension complemented cylindric algebras}, which is not only of a strictly
wider scope, but is actually much simpler.
Both cases reflect a {\it Henkin construction}, but in the case of polyadic algebras the procedure is much more complex. 
One starts with an algebra {\it dilates it}, meaning embedding it into a reduct of an algebra having infinitely extra dimensions,
fixes some of the extra dimensions obtaining a {\it free or rich extension of $\A$}, 
and then `constants' are stimulated as {\it algebraic endomorphism} on the dilated algebra,
and these endomorphisms
are used to eliminate cylindrifiers, witness \cite[p.449]{g}.

In cylindric algebras one also dilates the algebra, but then 
cylindrifiers are eliminated by the {\it spare dimensions}
via certain Boolean ultrafilters (which we call Henkin ultrafilters; that correspond exactly to Henkin's notion of {\it rich theories}). 
In this case a constant is not a complicated algebraic entity like an endomorphism, 
but it can be viewed as  simply {\it an index in the dilated dimension} which conforms more to Henkin's notion of expanding
the language by adding contstants or witnesses for existential formulas. 
This makes life much easier and also the construction lends itself
to more general contexts.

Indeed our results address possibly infinitary extensions of topological first order logic, and
Chang's modal logic.
We not only prove completeness and an 
omitting types theorem for such logics, but we also prove
an interpolation theorem, analagous to the Craig interpolation theorem
for first order logic, but in a more general setting. 

The results in \cite{g, g2},  are special cases of two of our three results proved for topological logic and Chang's modal logic.
The new interpolation theorem proved here which is not approached at all in the two cited references,
is next elaborated upon in a {\it universal algebraic way} as done
in {\it abstract algebraic logic.}

From the algebraic point of view, we depart from the so-called {\it locally finite} and {\it dimension complemented} algebras.
An algebra $\A$ is locally finite if the {\it dimension set } of every element in $\A$.
The dimension set of an element in $\A$ reflects the number of variables in the formula of the 
corresponding Tarski-Lindenbaum  algebra of formulas. 

An algebra is dimension complemented if the complement of the dimension set of every element is infinite; this reflects, 
in turn, that infinitely many variables lie outside the formula corresponding to the element, but the possibility
remains that this formula contains infinitely many variables, so such logics have an infinitary flavour. 
In fact, they can be seen as an instance of the so-called {\it finitary logics with infinitary predicates} \cite{ANS, BP, AGN, HMT2, Sayed}, 
finitary here, in turn, 
points out to the fact that quantification is only allowed
on finitely many variables, as is the case
with first order logic.

This is a natural generalization of first order logic, for in many classical theorems of first order logic, like 
Godel's completenes theorem, Craig interpolation theorem and the Orey-Henkin omitting 
types theorem, the proof does not depend on the fact that every formula contains many (free) variables
but rather on the weaker fact that {\it infinitely many variables} lie outside each formula, because
in such a case witnesses for existential formulas 
in Henkin constructions can, like the case with first order logic,  always be found.

But locally finite algebras, the algebraic counterpart
of topological predicate logic,  and for that matter the larger class of dimension complemented algebras, have
some serious defects when treated as the
sole subject of research in an autonomous algebraic theory.

In universal algebra one prefers to deal
with {\it equational classes} of algebras i.e. classes
of algebras characterized by
systems of postulates, in which
every postulate has the form of an equation (an identity).
Such classes are also referred to as {\it varieties}.

Classes of algebras which are not varieties are often
introduced in discussions
as specialized subclasses of varieties.
One often treats fields as a special case of rings.
This is due to the tradition that in algebra, mainly the equational
language and thus equational logic is used.
Thus, finding an equational form for an algebraic entity is always a value on its own
right.

Another reason for this preference, is
the fact that every variety is closed under
certain general closure operations frequently
used to construct new algebras from given ones.
We mean here the operations of forming subalgebras,
homomorphic images and direct products.
By a well known theorem of Garrett Birkhoff,
varieties are precisely those classes of algebras that have all
three of these closure properties.
Local finiteness does not have the form of an identity,
nor can it be equivalently
replaced by any identity or system of identities, nor indeed
any set of first order axioms.
This follows from the simple observation that
the ultraproduct of infinitely
many locally finite algebras   is not, in general,
locally finite, and a first order
axiomatizable class is necessarily closed under ultraproducts.
The same applies to the class of dimension complemented algebras.

The definition of local finiteness  contains an assumption
which considerably restricts the scope of the definition and thus
it is very tempting to just drop it, and se what happens. As is the case with Tarski's cylindric algebras, a lot does.
We hope, and in fact we think,  
that the reader will be  convinced of this bold declaration after reading the paper. 

Indeed, the restrictive character of this  notion
becomes obvious when we turn our attention to {\it cylindric set algebras}; these are concrete having having top element a 
{\it cartesian square},  namely, a set of the form $^{\alpha}U$, $\alpha$ an ordinal is the dimension; the Boolean operations 
are the usual operations
of intersection and complementation with
respect to $^{\alpha}U$ and cylindrifiers and diagonal elements are defined reflecting the semantics of existential quantifiers and equality.
If for $s, t\in {}^{\alpha}U$ and $i<\alpha$, $s\equiv_ i t$ means that $t(j)=s(j)$ for all $j\neq i$, 
then the {\it $i$th cylindrifier} is defined via 
$${\sf c}_iX=\{s\in {}^{\alpha}U: \exists t\in X (s\equiv_ i t)\}, X\subseteq {}^{\alpha}U,$$
and the $i, j$ diagonal via
 $${\sf d}_{ij}=\{s\in {}^{\alpha}U: s_i=s_j\}.$$
We find that there are such set algebras of all dimensions,
and set algebras that are not locally finite
are easily constructed. 

We  thereby {\it simply remove the condition of local finiteness} and also we will have occasion to deal with
topological cylindric algebras of finite dimension extending many deep results proved for cylindric algebras, and proving new ones. 
  
For the infinite dimensional case we study the {\it corresponding minimal algebraisable extension} 
of both predicate topological logic and Chang's modal logic, 
that necessarily allow infinitary predicates. The condition of local finiteness in the infinite 
dimensional case is not warranted from
the algebraic point of view because
it is a property that cannot be expressed by first order formulas, let alone equations or quasi-equations.

Roughly, minimal extension here means this (algebraizable) logic
corresponding to the quasi-variety generated by the class of algebras
arising from ordinary topological  predicate logic, namely, the class of locally  finite algebras. This correspondence is taken in the sense of
Blok and Pigozzi associating quasi-varieties to algebraizable logics \cite{BP}.

In algebraisable extensions of first order logic studied by Henkin, Monk and Tarski
and Blok and Pigozzi in \cite{HMT2, BP}, and even earlier 
by Andr\'eka and N\'emeti \cite{AGN},
the notion of a {\it formula schema} plays a key role. 
If we have a set of formulas $F$ say, then a formula schema is an element of $F$.
An instance of a formula schema
is obtained by substituting formulas for the formula variables, i.e for atomic formulas,
in this formula schema. A formula schema is called {\it type-free valid} if
all its instances are valid.
This is a new notion of validity defined in \cite[Remark 4.3.65]{HMT2}.

A drawback at least from the algebraic point of view for 
ordinary first order logic, and for that matter predicate topological logic 
is the following:
There are type-free valid formula schemas $\psi$, say of first order logic that
are not {\it uniformly} provable. Though each instance of $\psi$ is
provable,  these proofs vary from one instance
to the other. We cannot give a uniform proof of all these instances in spite
of there being a uniform cause $\psi$ of their validity.

The reason for this phenomena is that the standard formalism of first order logic
is {\it not} structural  in the sense of \cite{BP}.
In fact, this formalism is not even
{\it substitutional} in the sense  of \cite[definition 4.7 (ii) p.72.]{ANS}.
This means that a formula resulting from
{\it substituting} formulas for atomic formulas in any  valid formula, may not be valid.
To remedy this ``defect"  one can give a {\it structural formalism} of first order logic.

Following \cite{AGN, ANS} a logic is a quadruple $(F, \bold K, {\sf mng}, \models)$ where $F$ is a set (of formulas) in a certain signature,
$\bold K$ is a class of structure ${\sf mng}$ is a function with
domain $F\times \bold K$ and $\models \subseteq F\times F.$
Intuitively, $\bold K$ is the class of structures for our language
${\sf mng}(\phi, M)$ is the interpretation of $\phi$ in $M$, possibly relativized,
and $\models$ is the pure semantical
relation determined by $\bold K$. This of course is too broad a definition.
An algebraisable logic is defined next.   

\begin{definition} A logic  $(F, \bold K, {\sf mng}, \models)$ with formula algebra $\F$
of signature $t$ is {\it algebraizable} if
\begin{enumarab}

\item A set $Cn\L$ the logical connectives fixed and each $c\in Cn\L$ finite rank
determining  the signature $t,$

\item There is set $P$ called atoms such that $\F$ is the term algebra or
absolutely free algebra over $P$
with signature $t$,

\item ${\sf mng}_M=\langle {\sf mng}(\phi, M): \phi\in F\rangle \in Hom(\F)$,

\item There is a derived binary connective $\leftrightarrow$
and a nullary connective
$\top$ that is compatible with the meaning functions,
so that for all $\psi, \phi\in F$, we have ${\sf mng}(\phi)={\sf mng}(\psi)$ iff $M\models \phi\leftrightarrow\psi$
and $M\models \phi$ if $M\models \phi\leftrightarrow \top,$

\item For each $h\in Hom(\F,\F)$, $M\in \bold K$, there is an $N\in \bold K$
such ${\sf mng}_N={\sf mng}_M\circ h,$  so that validity is preserved by homomorphisms.
\end{enumarab}
\end{definition}

Item (5) is what guarantees that instances of valid formulas remain valid for a homomorphism applied to a formula
$\phi$ amounts to replacing the atomic formulas in $\phi$ by
formula schemes.
This is a crucial property for a logic to allow algebraization.

To form the algebraic counterpart of such a logic, which is a quasi-variety, there are essentially two conceptually different means.
One can define it  {\it syntactically using quasi-equations} via a Hilbert style axiomatization involving type free valid schemas that 
translate to quasi-equations in the signature $t$ \cite{ANS, BP}.
Or alternatively one can proceed semantically, defining the algebraic counterpart as the quasi-variety 
generated by the 'meaning algebras $\{{\sf mng}_M(\F): M\in \bold K\}$. 
These two notions in general are distinct, but in favourable circumstances they can coincide;
indeed this is the case when we have a completeness theorem
\cite{ANS}. Structural formalism of first order logic and non finite Hilbert-style complete axiomatizations go hand in hand.
Such issues will be approached in some depth below; where we show that this phenomena persists in the new topological 
context.

\subsection{Sample of results}

For the algebraisable version of 
topological logic we show that the corresponding algebraic counterpart, 
call it $V,$  is a not only a quasi-variety, but is in fact a 
variety, that is an expansion of the variety of representable 
cylindric algebras of infinite dimensions by interior operators. 

A plethora of results on representability 
and amalgamation for $V$ are proved. For example we show that the variety of representable algebras coincides with the class of algebras
having the {\it neat embedding property}, lifting a famous result of Henkin proved for cylindric algebras when we count in interior operators.

In universal algebra and indeed in the newly born field of universal logic a crucial 
and extremely fruitful role is played by the fact that certian {\it global} properties of {\it varieties}, like the variety $V$ above, 
typically amalgamation  properties are mirrored in corresponding
{\it local properties} of their free algebras, typically congruence extension properties and even equational consequence relations in the variety itself, 
which in turn corresponds to various forms of interpolation
when we happen to have an order, like the Boolean order, a condition that holds in our subsequent investigations.
The synthesis of these characterizations provides an illuminating and potentially very useful
bridge between the paradigms of algebra and 
logic, with results enriching both.

In this paper all results in the late \cite{MS}, on {\it interpolation}, 
{\it congruence extension properties} on free algebras and various forms of 
amalgamation on classes of 
algebras are obtained for $V.$ As a sample we show that the class of semi-simple algebras have the 
amalgamation property but $V$ itself does not, and the former result is equivalent to the fact the free algebras
satisfy a natural weak form of interpolation, call it $WIP$. From the second result we can infer that the free algebras {\it do not} satisfy the 
usual Craig interpolation property; in fact, it turns out that they do not  satisfy an interpolation property strictly weaker
than the Craig interpolation property, but of course strictly stronger than the $WIP$.
Sharp results on non-finite axiomatizability are obtained for several 
subvarieties of $V$ whose members have a 
{\it neat embedding property}, to be clarified below. 
Entirely analogous results are 
obtained for the variety corresponding to the algebraisable extension of Chang's predicate $\sf S4$ and 
$\sf S5$  modal logic.

We shall also show that several approximations of the variety of representable algebras cannot be 
axiomatized by a finite schema of equations. Such varieties are defined
via the notion of {\it neat reducts} an old venerable notion in the theory of cylindric algebras. Given $\alpha<\beta$, the 
{\it $\alpha$ neat reduct
of a $\beta$ dimensional algebra} 
is a subalgebra of the reduct of $\B$ obtained by discarding all operations indexed by $\beta\sim \alpha$ and keeping only
$\alpha$ dimensional elements. Denoting the class of topological cylindric algebra 
of dimension $\mu$ by $\TCA_{\mu}$, the $\alpha$ neat reduct of $\B\in \TCA_{\beta}$
is denoted by $\Nr_{\alpha}\TCA_{\beta}$; the latter is a $\sf TCA_{\alpha}$. 
A classical result of Monk (which we prove an analogue thereof for 
topological cylindric algebras) 
says that for cylindric algebras $\sf CA$s
If $\alpha>2$, $S\Nr_{\alpha}\CA_{\alpha+n}\neq \sf RCA_{\alpha}$ for all $n\in \omega$, where $\sf RCA_{\alpha}$ denotes the class of representable
$\CA_{\alpha}$s. 
On the other hand, a classical result of Henkin, which is a strong algebraic extention of Godel's completenes theorem, proved using a Henkin 
construction too, says
that $S\Nr_{\alpha}\CA_{\alpha+\omega}=\sf RCA_{\alpha}$, which we prove for $\TCA$s. 

We also prove the following 
result extending a recent result
of the present author and Robin Hirsch for 
several cylindrc-like algebras, namely:

\begin{theorem}\label{new} Let $\alpha>2$ be an  ordinal. Then for any $r\in \omega$, for any
finite $k\geq 1$, for any $l\geq k+1$ (possibly infinite),
there exist $\B^{r}\in S\Nr_{\alpha}\TCA_{\alpha+k}\sim S\Nr_{\alpha}\TCA_{\alpha+k+1}$ such
$\Pi_{r\in \omega}\B^r\in S\Nr_{\alpha}\TCA_{\alpha+l}$.
In particular, for any such $k$ and $l$, and for $\alpha$ finite, $S\Nr_{\alpha}\TCA_{\alpha+l}$ is not finitely axiomatizable over
$S\Nr_{\alpha}\TCA_{\alpha+k}$, and for infinite $\alpha$,  $S\Nr_{\alpha}\TCA_{\alpha+l}$ is not axiomatizable
by a finite schema over $S\Nr_{\alpha}\TCA_{\alpha+k}$.
\end{theorem}

In contrast we introduce another variety of {\it topological polyadic algebras of infinite dimensions}; the term
polyadic refers to the fact that the  signature of this new class contains all substitutions,  so is closer
to the polyadic paradigm, and  prove that such a variety can be axiomatized by a finite schema and it further 
enjoys the super amalgamation property.

\subsection{Product of modal logics}

We shall also deal rather extensively with 
topological logic with {\it only finitely many variables}, corresponding to finite dimensional topological
cylindric algebras of dimension $m$ say, with $m\in \omega$. 
Such a logic can be viewed as a 
predicate logic with $m$ variables enriched by $m$ modalities,
or as a propositional multi-dimensional modal logic with $2m$ modalities; call it $\L_m$. 

We show that for $m>2$ ($m$ finite), $\L_m$ is not finitely 
axiomatizable, it is undecidable, it is undecidable to 
tell whether a finite frame is a frame for  $\L_m$, $\L_m$ fails Craig interpolation and Beth definability,
and $\L_m$ fails the omitting types theorem 
in a very strong sense, even if we 
allow clique guarded semantics. 
We shall adress deeply decidability isues for such logics, by viewing them as {\it product}
modal logics.

One of the main reasons for the praise of modal logics in computer science
is their robust decidability, which is preserved under forming combinations of modal logics
like products, as long as there are no interaction axioms or constriants (fusions). This situation however
changes drastically as soon as some kind of interaction between the modalities is imposed. In fact, straightforward constructions
of combined modal logics from the simple 1-dimensional ones will almost certainly result in computionaly complex
logics. The fact that all three dimensional modal logics are undecidable can be intuitively
explained by the undecidability of the product $S5^3$ and its relation to
the undecidable fragment of
first order logic with $3$ variables; represented algebraically by $\sf CA_3$. 
But unlike $\CA_2$, even some two dimensional modal logics are undecidable, like products of transitive 
frames.  

Such a view will enable us to show that
unlike first order logic with two variables, the topological logic $\L_2$ with two variables 
is undecidable and does not have
the finite model property
It will then readily follows the equational theory of 
$\sf TRCA_2$ is undecidable, a significant point of deviation from cylindric algebras.

The `two dimensional undecidability result'  will be done by encoding tilings; 
that is encoding the $\mathbb{N}\times \mathbb{N}$ grid 
using the the two interior 
operators, which is the standard technique of proving undecidability
for many modal logics. But we will also show that $\L_2$ is finitely 
axiomatizable, but does not have the finite model property. 
There are refutable formulas that cannot be refuted in finite
Kripke models. The latter result holds  too for $\L_n$ when $n\geq 3$, but this is utterly unsurprsing.

Products of modal logics, like  temporal, spatial, epitemistic logics or multi-dimensional 
modal languages
interpreted in various product-like frames are very natural and clear formalisms
arising in both pure logic and numerios applications, like multi-agent systems. 
For example, dynamic topological logic dealt with earlier can be  interpreted 
semantically in products
of the form $(T, <)\times (W, R)$ where $(T, <)$ models the flow of time and $(W, R)$ is a frame for $\sf S4$ 
representing the topological space, with the $\sf S4$ box being also  
interpreted as the interior operator. By interpreting
$W$ as a domain of objects that can change
over time, one can view such product frames
as models for finite variable 
fragments of first order 
temporal and modal logics.

We shall also deal with guarded versions of $\L_m$ by relativizing the set of worlds or states, 
obtaining a finite variable fragment of predicate topological logic having 
nice modal 
behaviour. We show that such logics (with any number of finitely many varibales) is finitely axiomatizable, have the finite model property,
is decidable (in a strong way; in fact the universal theory of its modal algebras is decidable), and have 
the interpolation property.

\subsection{Concluding}

The algebraic facet of this paper can be seen as a refresher to proofs of many deep results proved for cylindric algebras,
and also {\it new ones} for cylindric algebras by passing to reducts of topological cylindric algebras by discarding the interior 
operators. 
Such results include the deep results of Andr\'eka \cite{Andreka} on the {\it complexity of universal axiomatizations} 
of the variety of representable cylindric algebras,
which lift {\it mutatis mutandis} to the `topological addition', 
the  answer to problem 2.12 in \cite{HMT1} given in \cite{HHbook}, together with its infinite analogue,
the main result in 
\cite{IGPL2} which is the solution to problem  4.4 in \cite{HMT2}, and all the results in \cite{Comer, AUU, MStwo, MS, AUamal, IGPL}
answering all  open 
problems in Pigozzi's landmark paper \cite{P} and more, several results in \cite{1}
confirming three conjectures of Tarski's on  cylindric algebras, 
formulated in the language of category theory, a theme initiated in 
\cite{conference}. 

We use advanced sophisticated machinery of cylindric algebra theory, like so called  rainbow constructions 
\cite{HH, HHbook, Hodkinson, HHbook2}, 
obtaining new results, strengthening results in \cite{ANT, HH, Hodkinson}
both algebraic and metalogical which we formulate for topological logics with finitely many 
variables, and finally we use tilings twice to prove undecidability of topological logics with more 
than one variable.

Relativizing states, we also deal with finite variable fragments of such topological logics as multi -modal logics, and guarded fragments
of finite variable predicate topological logic. 
Using games we show that such logics 
are finitely axiomatizable, and using a model-theoretic 
result of Herwig and the well-developed duality theory between Kripke frames and complex algebras, 
we show that such logics having $n$ variables, are also decidable; the universal theory of their  modal algebras
is decidable,  and have the definability properties of Beth and Craig, 
for each finite $n$.

Due to the length of the paper it is divided into four parts.
\begin{enumarab}
\item  Part one: {\bf Topological logic via cylindric algebras}. 

\item Part two: {\bf Amalgamation, interpolation and congruence extension properties in topological algebras}.
 
\item Part three: {\bf Logical consequences for extensions of predicate toplogical logic}. 

\item Part four: {\bf Logical consquences for finite variable fragments 
of first order logic}. 
\end{enumarab}

Each part can be read separately modulo cross references 
to other parts.

\subsection*{On the notation of all parts, some required basics in Topology}

We follow more or less standard  notation. But for the reader's convenience,
we include the following list of notation that will be used throughout the  paper.

An \textit{ordinal} $\alpha$ is transitive set (i.e., any member of $\alpha$ is
also a subset of $\alpha$) that is well-ordered by $\in$. Every
well-ordered set is order isomorphic to a unique ordinal. For
ordinals $\alpha, \beta$, $\alpha < \beta$ we means $\alpha \in
\beta$. An ordinal is therefore the set of all smaller ordinals, so
for a finite ordinal $n$ we have $n = \{ 0,1, \dots, n-1 \}$ and the least infinite ordinal is
$\omega = \{ 0,1, 2, \dots \}$.

A \textit{cardinal} is an ordinal not in bijection with any smaller
ordinal, briefly an {\it initial ordinal} and the cardinality $|X|$ of a set $X$ is the unique
cardinal in bijection with $X$. Cardinals are ordinals and are
therefore ordered by $<$ (i.e., $\in$). The first few cardinals are
$ 0 = \phi, 1, 2, \dots, \omega$ (the first infinite ordinal),
$\omega_1$ (the first uncountable cardinal). A set will be said to be
\textit{countable} if it has cardinality $\leq \omega$,
\textit{uncountable} otherwise, and \textit{countable infinite} if
it has cardinality $\omega$. $2^{\omega}$ denotes the power of the continuum. 

For a set $X$,  $\wp(X)$ denotes the set of all subsets of $X$, i.e. the powerset of $X$.
${}^AB$ denotes the set of functions from $A$ to $B$.
If $f\in {}^AB$ and $X\subseteq A$ then $f\upharpoonright X$
denotes the restriction of $f$
to $X$. We denote by $\dom f$ and $\rng f$ the domain and range of a given function
$f$, respectively. $A\sim B$ is the set $\{x\in A: x\notin B\}.$

We frequently identify a function $f$ with the sequence $\langle f_x:x \in \dom f\rangle$.
We write $fx$ or $f_x$ or $f(x)$ to denote the value of $f$ at $x$.
We define composition so that the righthand function acts first, thus
for given functions $f,g$, $f\circ g(x)=f(g(x))$, whenever the left hand side is defined, i.e
when $g(x)\in \rng f$.

For a non-empty set $X$, $f(X)$ denotes the image of $X$ under $f$, i.e
$f(X)=\{f(x):x\in X\}.$ If $X$ and $Y$ are sets then $X\subseteq_{\omega}Y$ denotes that $X$
is a finite subset of $Y$.

Algebras will be denoted by
gothic letters, and when we write $\A$ then we will be tacitly assuming 
that $A$ will denote  the universe
of $\A$.
However, in some occasions we will identify (notationally)
an algebra and its universe. 

If $U$ is an ultrafilter over $\wp(I)$ and if $\A_i$ is some structure (for $i\in I$)
we write either  $\Pi_{i\in I}\A_i/U$ or $\Pi_{i/U}\A_i$ for the ultraproduct of the $\A_i$ over $U$.
Fix some ordinal $n\geq 2$.
For $i, j<n$ the replacement $[i|j]$ is the map that is like the identity on $n$, except that $i$ is mapped to $j$ and the transposition
$[i, j]$ is the like the identity on $n$, except that $i$ is swapped with $j$.     We will refer to maps from $\tau:n\rightarrow n$ as transformations.
A transformation is  finite if
the set $\set{i<n:\tau(i)\neq i}$ is finite, so  if $n$ is finite then all transformations $n\rightarrow n$ are finite.
It is known, and indeed not hard to show, that any finite permutation is a product of transpositions
and any finite non-injective map is a product of replacements. A transformation is infinitary if it is not finite.

A \emph{Topological space} $\bold X$ is a pair $(X,\tau)$ where $X$ is a set and $\tau$ a collection of subsets of $X$ 
such that $\emptyset,X\in \tau$ and $\tau$ is closed under arbitrary unions and finite intersections. 
Such a collection is called a \emph{topology} on $X$ and its members are called \emph{open} sets. 
The complements of open sets are called \emph{closed} sets. Clearly, both $\emptyset, X$ are closed 
and arbitrary intersections and finite unions of closed sets are closed. For $A\subseteq X$, we denote by ${\sf int}A$ the {\it interior}
of $A$, which is the {\it largest open set} contained in $A$.

$\bold{X}=(X,\tau)$ is {\it discrete} if $\tau=\wp(X)$. Note that $\bold X$ is discrete if and only if ${\sf int} A=A$ for every
$A\subseteq X$. The space $\bold X$ is {\it almost discrete} if for all $A\in \tau$, ${\sf cl}(A)={\sf int}{\sf cl}(A),$
where ${\sf cl}(A)$, the {\it smallest closed set} containing $A$, is the {\it closure} of $A$.
Notice that the operations ${\sf int}$ and ${\sf cl}$ are {\it dual}; for a topological space with underlying set $X$, 
and $A\subseteq X$, we have ${\sf cl}(A)=X\sim[{\sf int}\sim  A].$

A set of the form $\bigcap_{n\in\mathbb{N}} U_n$, where $U_n$ are open sets, is called a $G_\delta$ set, 
and a set of the form $\bigcup_{n\in\mathbb{N}}F_n$, where $F_n$ are closed sets, is called an $F_\sigma$ set.

Let $X$ be the underlying set of a  topological space. A set $A\subseteq X$ is called \emph{nowhere dense} 
if its closure ${\sf cl}(A)$ has empty interior.
(This means equivalently that $X\sim{\sf cl}(A)$ is dense). 
So $A$ is nowhere dense iff ${\sf cl}(A)$ is nowhere dense. A set $A\subseteq X$ is \emph{meager or of first category} 
if $A$ is the countable union of nowhere dense sets. 
The complement of a meager set is called \emph{comeager}. 
So a set is comeager iff it contains the 
intersection of a countable family of dense open sets.

\begin{theorem}
Let $X$ be a topological space. The following statements are equivalent:
\begin{enumerate}
\item Every nonempty open set in $X$ is nonmeager.
\item Every comeager set in $X$ is dense.
\item The intersection of countably many dense open sets in $X$ is is dense.
\end{enumerate}
\end{theorem}
\begin{definition}
A topological space is called a \emph{Baire space} if it satisfies any of the equivalent conditions of the above proposition.
\end{definition}

\begin{theorem}[The Baire Category theorem]Every completely metrizable space is Baire. Every locally compact Hausdorff space is Baire. 
\end{theorem}

For operators on classes of algebras: ${\bf S}$ stands for the operation of forming subalgebras, ${\bf H}$ for
the operation of forming homomorphic images, ${\bf P}$ for the operation of
forming products, and $\bf Up$ for the operation of forming ultraproducts. In particular,
a class $\sf K$ is a variety iff ${\bf HSP}\sf K=\sf K$ and $\sf K$ is a quasi-variety if ${\bf SPUp}\sf K=\sf K$.

\section{Basics}

Let $\alpha$ be an arbitrary ordinal $>0$. Cylindric set algebras are algebras whose elements are relations of a certain 
pre-assigned arity, the dimension, endowed with set-theoretic operations
that utilize the form of elements of the algebra as sets of sequences.
$\B(X)$ denotes the Boolean set algebra $\langle \wp(X), \cup, \cap, \sim, \emptyset, X\rangle$.
Let $U$ be a set and $\alpha$ an ordinal; $\alpha$ will be the dimension of the algebra.
For $s,t\in {}^{\alpha}U$ write $s\equiv_i t$ if $s(j)=t(j)$ for all $j\neq i$.
For $X\subseteq {}^{\alpha}U$ and $i,j<\alpha,$ let
$${\sf c}_iX=\{s\in {}^{\alpha}U: \exists t\in X (t\equiv_i s)\}$$
and
$${\sf d}_{ij}=\{s\in {}^{\alpha}U: s_i=s_j\}.$$

$\langle \B(^{\alpha}U), {\sf c}_i, {\sf d}_{ij}\rangle_{i,j<\alpha}$ is called the {\it full cylindric set algebra of dimension $\alpha$
with unit (or greatest or top element) $^{\alpha}U$}. $^{\alpha}U$ is called a {\it cartesian space}.

Examples of subalgebras of such set algebras arise naturally from models of first order theories.
Indeed, if $\M$ is a first order structure in a first
order language $L$ with $\alpha$ many variables, then one manufactures a cylindric set algebra based on $\M$ as follows.
Let
$$\phi^{\M}=\{ s\in {}^{\alpha}{\M}: \M\models \phi[s]\},$$
(here $\M\models \phi[s]$ means that $s$ satisfies $\phi$ in $\M$), then the set
$\{\phi^{\M}: \phi \in Fm^L\}$ is a cylindric set algebra of dimension $\alpha$. Indeed
\begin{align*}
\phi^{\M}\cap \psi^{\M}&=(\phi\land \psi)^{\M},\\
^{\alpha}{\M}\sim \phi^M&=(\neg \phi)^{\M},\\
{\sf c}_i(\phi^{\M})&=\exists v_i\phi^{\M},\\
{\sf d}_{ij}&=:(x_i=x_j)^{\M}.
\end{align*}

Instead of taking ordinary set algebras, 
as in the case of cylindric algebras, with units of the form $^{\alpha}U$, one 
may require {\it that the base $U$ is endowed with some topology}. This enriches the algebraic structure. 
For given such an algebra, 
for each $k<\alpha$, one defines an {\it interior operator} on $\wp(^{\alpha}U)$ by
$$I_k(X)=\{s\in {}^{\alpha}U; s_k\in {\sf int}\{a\in U: s_a^k\in X\}\}, X\subseteq {}^{\alpha}U.$$
Here $s_a^k$ is the sequence that agrees with $s$ except possibly at  $k$ where its value is $a$.
This gives a {\it topological cylindric set algebra of dimension $\alpha$}.
The dual operation of $I_k$ is ${\sf Cl}_k$ defined by
$${\sf Cl}_k(X)=\{s\in {}^{\alpha}U; s_k\in {\sf Cl}\{a\in U: s_a^k\in X\}\}, X\subseteq {}^{\alpha}U.$$
Notice that when $U$ has the indiscrete topology, then ${\sf Cl}_k(X)={\sf c}_kX$. 

A more general semantics is provided by the {\it Chang systems}:
\begin{definition}
A {\it Chang system} is a pair $(U, V)$, where $U$ is a non-empty set 
and $$V:U\to \wp(\wp(U)).$$
\end{definition}
Given such a system, one can introduce unary operations, called {\it box operators}
on $\wp({}^{\alpha}U)$ as follows:
$$s\in \Box_iX\Longleftrightarrow \{u\in U: {\sf s}^i_u\in X\}\in V(s_i), X\subseteq {}^{\alpha}U.$$

The interior operators, as well as the box operators can also be defined on {\it weak spaces}, 
that is, sets of sequences agreeing co-finitely with a given fixed sequence. This makes a difference only when $\alpha$ is infinite.
We mention the case of interior operators, the box operators are defined entirely 
analogously  using Chang systems.

A {\it weak space of dimension $\alpha$} is a set of the 
form $\{s\in {}^{\alpha}U: |\{i\in \alpha: s_i\neq p_i\}|<\omega\}$ 
for a given fixed in advance $p\in {}^{\alpha}U$.
Now for $k<\alpha$, define
$$I_k(X)=\{s\in {}^{\alpha}U^{(p)}: \{s_k\in {\sf int}\{u\in U: s_k^u\in X\}\}.$$

But we can even go further. Such operations 
also extend to the class of representable algebras $\CA$s, briefly $\sf RCA_{\alpha}$. $\sf RCA_{\alpha}$ is defined to be the class
${\bf SP}\sf Cs_{\alpha}$. This class is also equal to ${\bf SP}\sf Ws_{\alpha}$, and 
it is known that $\sf RCA_{\alpha}$,  is a variety, hence closed under $\bf H$, 
though infinitely many schema of 
equations are required to axiomatize it \cite{Andreka}, witness also theorem 
\ref{neat} below.  

An algebra  in $\sf RCA_{\alpha}$ is isomorphic to a set algebra with universe $\wp(V)$;
the top element $V$ is a {\it generalized} space which is a set of the form $\bigcup_{i\in I}{}^{\alpha}U_i$, $I$ a set $U_i\neq \emptyset$ $(i\in I)$,
and  $U_i\cap U_j=\emptyset$ for $i\neq j$. The class of all such concrete algebras is denoted by $\sf Gs_{\alpha}$. We refer to $\A\in {\sf Gs}_{\alpha}$
as a {\it generalizd set algebra of dimension $\alpha$}.

So let $\A\in {\sf RCA}_{\alpha}$, and assume that $\A\cong \B$ where  
$\B\in \sf Gs_{\alpha}$ 
has top element  the  generalised space $V$.
The base of $V$ is the set $U=\bigcup_{s\in V}\rng s.$
Then one defines the interior operator $I_k$ on $\B$ by:
$$I_k(X)=\{s\in V: s_k\in {\sf int}\{a\in U: s_a^k\in X\}\}, X\subseteq V.$$
and, for that matter  the box operator relative to a Chang system $V:U\to \wp(\wp(U))$ as follows
$$s\in \Box_k(X)\Longleftrightarrow \{a\in U: s_a^k\in X\}\in V(s_k), X\subseteq V.$$

The following lemma is very easy to prove, so we omit the proof.
Formulated only for set algebras, it also holds for weak set algebras. 
\begin{lemma}\label{box}  For any  ordinal $\mu>1$, $\A\in {\sf Cs}_{\mu}$ and  $k<\mu$, 
let $I_k$ and $\Box_k$ be as defined above. Then
if $\A\in \sf Cs_{\alpha}$ has top element $^{\alpha}U$ and $\beta>\alpha$, then 
the following hold for any $Y\subseteq {}^{\alpha}U$ and $k<\alpha:$
\begin{enumarab}
\item  $I_k(Y)\subseteq Y,$ $\Box_k(Y)\subseteq Y,$ 

\item  If $f: \wp(^{\alpha}U)\to {}\wp(^{\beta}U)$ is defined via
$$X\mapsto \{s\in {}^{\beta}U: s\upharpoonright \alpha\in X\},$$
then $f(I_kX)=I_k(f(X))$ and $f(\Box_kX)=\Box_k(f(X))$,
for any $X\subseteq {}^{\alpha}U$.
\end{enumarab}
\end{lemma}

For an algebra $\A$ in $\sf Gs_{\alpha}$ with top element $V$ and base $U$, let $\A^{t}$ be the $\sf TPCA_{\alpha}$ 
obtained when $U$ is endowed with the some topology $\tau$, 
equivalently the Chang algebra (this will turn out to be an $S5$ Chang algebra, to be defined shortly) 
corresponding to the Chang system
$F:U\to \wp(\wp (U))$, defined via $F(x)=\tau$ $(x\in U)$.

\begin{lemma} Let $\A, \B$ be in $\sf Gs_{\alpha}$ such that $\A\subseteq \B$ 
have the same top element giving the same base $U$. If $U$ is endowed with any topology;  then 
$\A^{t}\subseteq \B^{t}$. If $\A\cong \B$, then $\A^t\cong \B^t$.
\end{lemma}
 
In cylindric algebra theory a subdirect product of set algebras is isomorphic to a generalized set algebra.
We show that this phenomena persists when the bases carry topologies; we need to describe the topology on the base 
of the resulting generalized set algebra in terms
of the topologies on the bases of the set algebras 
involved in the subdirect product.

\begin{definition} Let $\{X_i:i\in I\}$ be a family of topological spaces indexed by $I$. 
Let $X=\bigcup X_i$ be the disjoint union of the underlying sets. For each $i\in I$ let
$\phi_i: X_i\to X$ be the canonical injection. 
The {\it coproduct on $X$} is defined as the finest topology on $X$ for which the canonical injections are continuous.
\end{definition}
That is a  subset $U$ of $X$ is open in the coproduct topology on $X$ iff its preimage $\phi_i^{-1}(U)$ is open in $X_i$ for each $i\in I$ iff 
iff its intersection with $X_i$ is open relative to $X_i$ for each $i\in I$.

\begin{theorem} Let $\B$ be the $\sf Gs_{\alpha}$ with unit $V=\bigcup_{i\in I}{}^{\alpha}U_i$ where $U_i\cap U_j=\emptyset$ and base 
$\bigcup_{i\in I}U_i$ carrying a topology.  Assume that $\B$ has universe $\wp(V)$. 
Let $\A_i$ be the $\sf Cs_{\alpha}$ 
with base $U_i$, $U_i$ having the subspace topology and universe $\wp(^{\alpha}U_i)$. 
Then $f:\B\to \prod_{i\in I}\A_i$ 
defined by
$X\mapsto (X\cap {}^{\alpha}U_i:I\in I)$ is an isomorphism of cylindric algebras; 
furthermore it respects the interior operators stimulated by the topologies
on the bases.
\end{theorem}
${\sf TCs}_{\alpha}({\sf TGs}_{\alpha})$ 
denotes the class of topological (generalized) set algebras.

\begin{theorem} ${\bf SP}{\sf TCs}_{\alpha}\subseteq {\sf TGs}_{\alpha}.$
\end{theorem}
\begin{proof} Suppose that $\C\subseteq \prod_{i\in I}\D_i$ each $\D_i\in {\sf TCs}_{\alpha}$ 
with base $U_i\neq 0$, and $U_i\cap U_j=\emptyset$. Each $U_i$ has a topology.
Let $f$ be as in the previous theorem. 
Then $f^{-1}\upharpoonright \C$ is an isomorphism into a $\sf TGs_{\alpha}$ 
whose base $\bigcup_{i\in I}U_i$ carry the coproduct 
topology.
\end{proof}

Now such algebras lend itself to an abstract formulation aiming to capture the concrete set algebras; or rather the variety 
generated by them.

This consists of expanding the signature of cylindric algebras 
by unary operators,  or modalities, one for each $k<\alpha$, satisfying certain identities.

The axiomatizations we give are actually simpler than  those stipulated by Georgescu in \cite{g, g2}, 
although  locally finite polyadic algebras and locally finite cylindric algebras are  equivalent.
We use only substitutions corresponding to replacements; in the case of dimension complemented algebras
all substitutions corresponding to finite transformations are term 
definable from these \cite{HMT1}. This makes axiom $(A8)$ on p.1 of 
\cite{g} superfluous.

In \cite{g, g2}  representation theorems are proved  for locally finite polyadic algebras; 
here we extend this theorem in three ways. We prove a strong representation theorem for {\it dimension complemented algebras}
this is a strictly larger class.  The logic corresponding to such algebras allow infinitary predicates.
We prove an interpolation and an omitting types for such logics, too.
The constructions used are standard Henkin constructions; for luckily the `expanded' 
semantics allows such 
proofs.

We start with the standard definition of cylindric algebras \cite[Definition 1.1.1]{HMT1}:

\begin{definition}
Let $\alpha$ be an ordinal. A {\it cylindric algebra of dimension $\alpha$}, a  $\CA_{\alpha}$ for short, 
is defined to be an algebra
$$\C=\langle C, +, \cdot, -, 0, 1, \cyl{i}, {\sf d}_{ij}  \rangle_{i,j\in \alpha}$$
obeying the following axioms for every $x,y\in C$, $i,j,k<\alpha$

\begin{enumerate}

\item The equations defining Boolean algebras,

\item $\cyl{i}0=0,$

\item $x\leq \cyl{i}x,$

\item $\cyl{i}(x\cdot \cyl{i}y)=\cyl{i}x\cdot \cyl{i}y,$

\item $\cyl{i}\cyl{j}x=\cyl{j}\cyl{i}x,$

\item ${\sf d}_{ii}=1,$

\item if $k\neq i,j$ then ${\sf d}_{ij}=\cyl{k}({\sf d}_{ik}\cdot {\sf d}_{jk}),$

\item If $i\neq j$, then ${\sf c}_i({\sf d}_{ij}\cdot x)\cdot {\sf c}_i({\sf d}_{ij}\cdot -x)=0.$

\end{enumerate}
\end{definition}
For a cylindric algebra $\A$, we set ${\sf q}_ix=-{\sf c}_i-x$ and ${\sf s}_i^j(x)={\sf c}_i({\sf d}_{ij}\cdot x)$.
%We consider only infinite dimensional cylindric algebras.
Now we want to abstract equationally the prominent features of the concrete interior operators defined on cylindric 
set and weak set algebras.
We expand the signature of $\CA_{\alpha}$ 
by a unary operation $I_i$ 
for each $i\in \alpha.$ In what follows $\oplus$ denotes the operation of symmetric difference, that is, 
$a\oplus b=(\neg a+b)\cdot (\neg b+a)$.
For $\A\in \CA_{\alpha}$ and $p\in \A$, $\Delta p$, {\it the dimension set of $p$}, is defined to be the set 
$\{i\in \alpha: {\sf c}_ip\neq p\}.$ In polyadic terminology $\Delta p$ is called {\it the support of $p$}, and if $i\in \Delta p$, then $i$ 
is said to {\it  support $p$} \cite{g,g2}.

\begin{definition}\label{topology} A {\it topological cylindric algebra of dimension $\alpha$}, $\alpha$ an ordinal,
is an algebra of the form $(\A,I_i)_{i<\alpha}$ where $\A\in \sf CA_{\alpha}$ 
and for each $i<\alpha$, $I_i$ is a unary operation on 
$A$ called an {\it  interior operator} satisfying the following equations for all $p, q\in A$ and $i, j\in \alpha$:
\begin{enumerate}
\item ${\sf q}_i(p\oplus q)\leq {\sf q}_i (I_ip\oplus I_iq),$
\item $I_ip\leq p,$
\item $I_ip\cdot I_ip=I_i(p\cdot q),$
\item $p\leq I_iI_ip,$
\item $I_i1=1,$
\item ${\sf c}_kI_ip=I_ip, k\neq i, k\notin \Delta p,$
\item ${\sf s}_j^iI_ip=I_j{\sf s}_j^ip, j\notin \Delta p.$ 
\end{enumerate}
\end{definition}
The class of all such topological cylindric  
algebras are denoted by ${\sf TCA}_{\alpha}.$ 

We do the same task axiomatizing the properties of Chang's modal operators, or boxes, equationally.

\begin{definition} A {\it Chang cylindric algebra of dimension $\alpha$}, $\alpha$ an ordinal,  
is an algebra of the form $(\A, \Box_i)_{i\in \alpha}$ where $\A\in \CA_{\alpha}$
and for each $i<\alpha$, $\Box_i$ is a unary operator on $\A$, called a {\it modality}, 
satisfying the following equations for all $p,q\in A$ and $i,j\in \alpha$.
\begin{enumerate}
\item ${\sf q}_i(p\oplus q)\leq {\sf q}_i (\Box_ip\oplus \Box_iq),$
\item ${\sf s}_j^i\Box_ip=\Box_j{\sf s}_j^ip, j\notin \Delta p.$ 
\end{enumerate}
\end{definition}
Consider the following equations expressible in the 
signature of Chang algebras of dimension $\alpha$; where $i\in \alpha$:
\begin{enumerate}
\item $\Box_i1=1,$
\item $\Box_ip\leq p,$
\item $\Box_ip\cdot \Box_ip=\Box_i(p\cdot q),$
\item ${\sf c}_k\Box_ip=\Box_ip, k\neq i, k\notin \Delta p,$
\item $\Box_ip\leq \Box_i\Box_ip,$
\item $\neg \Box_i\neg p\leq \Box_i\neg \Box_i\neg p.$
\end{enumerate}
The {\it $\sf S4$ Chang algebras of dimension $\alpha$} are defined as the Chang algebras of dimension $\alpha$ 
with properties equivalent to items $(1)-(5)$ and the {\it $S5$ Chang algebras of dimension $\alpha$} are the 
Chang cylindric algebras of dimension $\alpha$ satisfying items $(1)-(6)$. 
Notice that the $\sf S4$ Chang algebras are equivalent to the topological cylindric algebras of the same
dimension.
For $\B=(\A, I_i)_{i<\alpha}\in {\sf TCA}_{\alpha}$ we write
$\Rd_{ca}\B$ for $\A$. Notice too that every $\CA_{\alpha}$ can be extended to a 
$\sf TCA_{\alpha}$, by defining for all $i<\alpha$,
$I_i$ to be the identity function. 

Topological algebras  in the form we defined are {\it not } 
Boolean algebras with operators because the interior operators do not distribute
over the Boolean join. 

But we could have just as well worked with the {\it dual operators}, in which case
we land in the realm of Boolean algebras with operators. From the point of view of multi modal logic such 
operators  are the diamonds and the interior operators are the boxes.

But in all cases algebras dealt with are {\it not completely additive}; 
cylindrifiers are completely additive but the interior operators 
are not as shown next.
\begin{example}

Let $\A= \wp(^{\omega}\mathbb{N})$ with the co-finite topology on $\mathbb{N}$. 
Let $X_n=\{n\}\times {}^{\omega}\mathbb{N}$. Then 
$$I_0X_n=\emptyset,$$ 
and so $$\bigcup I_0X_n=\emptyset$$
But $$\bigcup_{n\in \omega}X_n=\mathbb{N}$$ 
hence
$$I_0(\bigcup X_n)=\mathbb{N}\neq \bigcup I_0X_n.$$

\end{example}

Second observation is that the interior operators are not term definable, for if $U$ is an infinite set and $\A$ is the full set algebra
with base $U$ of dimension $\alpha>1$, then if $U$ has the discrete topology and $i<\alpha$, then $I_iX=X$ for any $X\in \A$, which is not the case
when $U$ has the indiscrete topology. 
In other words the cylindric structure does not uniquely define the interior operators.

We do not know whether one can construct a set algebra with base $U$ and two non-homeomorphic topologies on $U$ such that
the induced interior operators gives rise to isomorphic 
topological cylindric set algebras.

If $\A$ is a set algebra with base $U$, this  is concretely reflected by giving $U$ the 
discrete topology.  Viewed otherwise, if $\A$ is in $\sf RCA_{\alpha}$, then this expansion is also representable by giving the base $U$
of the $\sf Gs_{\alpha}$ representing $\A$
the discrete topology. This simple observation will turn out immensely 
useful to obtain results about $\TCA_{\alpha}$ by bouncing it back to the cylindric part. 
This works in the case of transferring {\it negative results} for cylindric algebras to the topological paradigm, 
but does not help much in case we are encountered with a 
positive result. For example as we shall see, though the equational theory of $\sf RCA_2$ is known to be decidable, 
it will turn out that the equational theory of the class of representable topological cylindric algebras of dimension $2$ is not.

Such an  observation also holds for $S5$ Chang algebras, too because a 
discrete space is obviously almost discrete.

%Unless otherwise specified $\alpha$, the dimension, will be an infinite ordinal and we shall address only ${\sf TCA}_{\alpha}$.
%We will have occasion to allow finite dimensions. 

{\it We stipulate that each and every result, with no single exception, proved for $\sf TCA_{\alpha}$ can be obtained using the same methods for 
Chang algebras, $\sf S4$ Chang algebras and $S5$ Chang algebras.}

An algebra $\B$ is {\it locally finite (dimension complemented)}
if $\Rd_{ca}\B$ is such. 

We denote by $\sf TLf_{\alpha}$ and $\sf TDc_{\alpha},$ the classes of {\it locally finite and dimension complemented cylindric topological 
algebras of dimension $\alpha$},
respectively. That is, $\B\in \sf TDc_{\alpha},$ if $\Delta x\neq \alpha$ 
for every $x\in B$; this turns out, in the infinite dimensional case, equivalent to $\alpha\sim \Delta x$ is infinite for every
$x\in B$. On the other hand, 
$\B\in \sf TLf_{\alpha}$ if $\Delta x$ is finite for all $x\in B$ (recall that $\Delta x=\{i\in \alpha: {\sf c}_ix\neq x\}$).
For finite dimension obviously every algebra is locally finite.

We also need the notion of {\it compressing} dimensions 
and, dually,  {\it dilating them}; expressed by the notion of neat reducts.

\begin{definition} 
\begin{enumarab}
\item Let $\alpha<\beta$ be ordinals and $\B\in \sf TCA_{\beta}$. Then $\Nr_{\alpha}\B$ is the algebra with universe 
$Nr_{\alpha}\A=\{a\in \A: \Delta a\subseteq \alpha\}$ and operations obtained by discarding the operations of $\B$ 
indexed by ordinals in $\beta\sim \alpha$. 
$\Nr_{\alpha}\B$ is called the {\it neat $\alpha$ reduct of $\B$}. If $\A\subseteq \Nr_{\alpha}\B$, with $\B\in \sf TCA_{\beta}$,
then we say that $\B$ is  {\it a $\beta$
dilation of $\A$}, or simply {\it a dilation} of $\A$.
\item An {\it injective} homomorphism $f:\A\to \Nr_{\alpha}\B$ is called a {\it neat embedding}; if such an $f$ exists, then we say
that $\A$ {\it neatly embeds into its dilation $\B$}. In particular, if $\A\subseteq \Nr_{\alpha}\B,$ then $\A$ neatly embeds into
$\B$ via the inclusion map. 
\end{enumarab}
\end{definition}
Note that the algebra $\Nr_{\alpha}\B$ is well defined; it is closed under the cylindric operations; this is well 
known and indeed easy to show, and it also closed under all the {\it interior operators} $I_i$ for $i<\alpha$, for
if $x\in \Nr_{\alpha}\B$, and $k\in \beta\sim \alpha$, then by axiom (6) of definition \ref{topology}, 
$k\notin \alpha\supseteq \Delta x\cup \{i\}\supseteq \Delta(I_i(x))$, 
hence ${\sf c}_k(I_i(x))=I_i(x).$

A piece of {\it notation} used throughout. If $\A$ is an algebra and $X\subseteq \A$, then $\Sg^{\A}X$ denotes the subalgebra
of $\A$ generated by $X$.

\begin{theorem}\label{dilate} Let $\alpha\geq \omega$. If $\A \in {\sf TDc}_{\alpha}$ and $\beta>\alpha$, then there exists
$\B\in \sf TCA_{\beta}$ such that $\A\subseteq \Nr_{\alpha}\B$ and for all $X\subseteq \A$, 
$\Sg^{\A}X=\Nr_{\alpha}\Sg^{\B}X$.
\end{theorem} 
\begin{proof} Exactly like the proof in \cite[Theorem 2.6.49]{HMT1} defining the interior operators the obvious way. 
\end{proof}

Recall that for a class $\sf K$, ${\bf S}$ stands for the operation of forming subalgebras of $\sf K$, and ${\bf P}\sf K$ 
that of forming direct products.

\begin{definition}
Let $\delta$ be a cardinal. Let $\alpha$ be an ordinal.
Let$_{\alpha} \Fr_{\delta}$ be the absolutely free algebra on $\delta$
generators and of type $\TCA_{\alpha}.$ For an algebra $\A,$ we write
$R\in \Co\A$ if $R$ is a congruence relation on $\A.$ Let $\rho\in
{}^{\delta}\wp(\alpha)$. Let $L$ be a class having the same
signature as $\TCA_{\alpha}.$ Let
$$Cr_{\delta}^{(\rho)}L=\bigcap\{R: R\in {\sf Co}_{\alpha}\Fr_{\delta},
{}_{\alpha}\Fr_{\delta}/R\in \mathbf{SP}L, {\mathsf
c}_k^{_{\alpha}\Fr_{\delta}}{\eta}/R=\eta/R \text { for each }$$
$$\eta<\delta \text
{ and each }k\in \alpha\smallsetminus \rho(\eta)\}$$ and
$$\Fr_{\delta}^{\rho}L={}_{\alpha}\Fr_{\beta}/Cr_{\delta}^{(\rho)}L.$$
\end{definition}
The ordinal $\alpha$ does not figure out in $Cr_{\delta}^{(\rho)}L$
and $\Fr_{\delta}^{(\rho)}L$ though it is involved in their
definition. However, $\alpha$ will be clear from context so that no
confusion is likely to ensue.

\begin{definition} Assume that $\delta$ is a cardinal, $L\subseteq \TCA_{\alpha}$, $\A\in L$,
$x=\langle x_{\eta}:\eta<\beta\rangle\in {}^{\delta}A$ and $\rho\in
{}^{\delta}\wp(\alpha)$. We say that the sequence $x$ {\it $L$-freely
generates $\A$ under the dimension restricting function $\rho$}, or
simply $x$ freely generates $\A$ under $\rho,$ if the following two
conditions hold:
\begin{enumroman}
\item $\A=\Sg^{\A}\rng x$ and $\Delta^{\A} x_{\eta}\subseteq \rho(\eta)$ for all $\eta<\delta$.
\item Whenever $\B\in L$, $y=\langle y_{\eta}, \eta<\delta\rangle\in
{}^{\delta}\B$ and $\Delta^{\B}y_{\eta}\subseteq \rho(\eta)$ for
every $\eta<\delta$, then there is a unique homomorphism $h$ from
$\A$ to $\B$ such that $h\circ x=y$.
\end{enumroman}
\end{definition}

It can be proved without much difficulty that in the above characterization the existence of a unique homomorphism $h$ from $\A$ to $\B$ such
that $h\circ x=y$ can be replaced by the existence of a unique {\it injective} homomorphism $h$ from $\A$ to $\B$ such
that $h\circ x=y$.

\begin{lemma}\label{prep} Let $\alpha\geq \omega$ and let $\rho:\mu\to \wp(\alpha)$ such that $\Fr_{\mu}^{\rho}\TCA_{\alpha}\in \sf TDc_{\alpha}$
Then for any ordinal $\beta>\alpha$,
the sequence $x=\langle \eta/Cr_{\mu}^{\rho}\MA_{\beta}:
\eta<\mu\rangle$ $\TCA_{\alpha}$ - freely generates
$\Nr_{\alpha}\Fr_{\mu}^{\rho}(\TCA_{\beta})$ under $\rho$.
\end{lemma}
\begin{proof}
Let $\C\in \TCA_{\alpha}$ and let $y:\mu\to \C$ be a homomorphism
such that $\Delta y_{\eta}\subseteq \rho \eta$ for all $\eta<\mu$.  Then we can assume that $\rng y$
generates $\C$, so that $\C\in \Dc_{\alpha}$, hence $\C\in \Nr_{\alpha}\TCA_{\beta}$.
Accordingly, let $\C'\in \TCA_{\beta}$ be such that $\C=\Nr_{\alpha}\C'$. Then clearly $y\in
{}^{\mu}\C'$ and $\Delta y_{\eta}\subseteq \alpha$ for all
$\eta<\mu$. Let $\D=\Fr_{\mu}^{\rho}(\TCA_{\beta})$. Then by freeness
there exists  a homomorphism $h$ from $\D$ to $\C'$
such that $h\circ x=y$. Clearly $h$ is a homomorphism from $\Rd_{\alpha}\D$ to $\Rd_{\alpha}\C'$,
hence it is a homomorphism from $\Sg^{\Rd_{\alpha}\D}\rng x$ to $\Sg^{\Rd_{\alpha}\C'}h(\rng(x))$. Since
$\rng x\subseteq \Nr_{\alpha}\D$, we have $h$ is a homomorphism from
$\Sg^{\Nr_{\alpha}\D}\rng x=\Nr_{\alpha}\Sg^{\D}\rng x$
to $\C$, such that $h(\eta/Cr_{\mu}^{\rho}\TCA_{\beta})= a_{\eta}$ and we are done.
In particular, we have
$\Nr_{\alpha}\Fr_{\mu}^{\rho}\TCA_{\beta}\cong \Fr_{\mu}^{\rho}(\TCA_{\alpha}).$
\end{proof}

\section{Completeness, Interpolation and Omitting types}

In this section $\alpha$ will be an infinite ordinal. 
To prove our first (completeness) theorem, we formulate and prove several lemmas. Properties of substitutions reported in \cite{HMT1} 
are freely used. For example, for every $\A\in \sf TDc_{\alpha}$ and every finite transformation $\tau$ we have a unary operation
${\sf s}_{\tau}$ that happens to be a Boolean endomorphism on $\A$ \cite[Theorem 1.11.11]{HMT1}.

\begin{lemma}\label{diagonal} 
\begin{enumarab}
\item Let $\C\in \CA_{\alpha}$ and let $F$ be a Boolean filter on $\C$.
Define the relation $E$ on $\alpha$ by $(i,j)\in E$ if and only if
${\sf d}_{ij}\in F$. Then $E$ is an equivalence relation on $\alpha$.
\item Let $\C\in \CA_{\alpha}$ and $F$ be a Boolean filter of $\C$. Let $V=\{\tau\in {}^{\alpha}\alpha: |\{i\in \alpha: \tau(i)\neq i\}|<\omega\}$.
For $\sigma, \tau\in V$,  write 
$$\sigma\equiv_E\tau\textrm {  iff } 
(\forall i\in \alpha) (\sigma (i),\tau(i))\in E.$$ 
and let 
$$\bar{E}=\{(\sigma,\tau)\in {}^2V: \sigma\equiv_E \tau\}.$$
Then $\bar{E}$ is an euiqvalence relation on $V$. Let $W= V/\bar{E}.$ 
For $h\in W,$ write $h=\tau/\bar{E}$ for $\tau\in V$ such that
$\tau(j)/E=h(j)$ for all $j\in \alpha$. 
Let 
$f(x)=\{ \bar{\tau} \in W: {\sf s}_{\tau}x\in F\}.$  
Then $f$ is well defined. 

Furthemore, $W$ can be identified with the weak space $^{\alpha}[U/E]^{(\bar{p})}$ where $\bar{p}=(p(i)/E: i <\alpha)$ 
via $\tau/\bar{E}\mapsto [\tau]$, 
where $[\tau](i)=\tau(i)/E.$ Accordingly, we write $W={}^{\alpha}[U/E]^{(\bar{p})}.$ 
\end{enumarab}
\end{lemma}

\begin{definition} Let $\A$ be an algebra having a cylindric reduct of dimension $\alpha$. 
A Boolean ultrafilter $F$ of $\A$ is said to be {\it Henkin} if for all $k<\alpha$, for all $x\in A$, whenever
${\sf c}_kx\in F$, then there exists $l\notin \Delta x$ such
that ${\sf s}_l^kx\in F$.
\end{definition}

\begin{lemma} Let everything be as in the previous lemma, and assume that $F$ is a Henkin ultrafilter.
Then $f$ as defined in the previous lemma is a $\sf CA$ homomorphsim.
\end{lemma}
\begin{proof}\cite{IGPL2}.
\end{proof}
\begin{definition}\label{interior}
Let everything be as in the hypothesis of lemma \ref{essence}.
For $s\in W$ and $k<\alpha$ we write $s^k_u$ for $s^k_{u/E}$. 
For $k\in \alpha,$ then  $I_k$ is the (interior) operator  on $\wp(W)$ 
defined by $I_k(X)=\{s\in W: s_k\in {\sf int}\{u\in U: s^k_{u}\in X\}\}.$  Similarly, if $V: U/E\to \wp(\wp(U/E))$ is a Chang system
then $\Box_k$ is defined on $\wp(W)$ by  
$s\in \Box_k(X)\Longleftrightarrow \{u/E\in U/E: {s}^k_{u}\in X\}\in V[s(i/E)].$
\end{definition}

\begin{lemma}\label{essence}
\begin{enumarab}
\item  Assume that $\C\in \sf TDc_{\alpha}$, $F$ is a Henkin  ultrafilter of $\C$ and $a\in F.$
Then there exist a non-empty set $U$, $p\in {}^{\alpha}U,$
a topology on $U/E$ and a homomorphism 
$f:\C\to (\wp(W), I_i)_{i<\alpha}$ with $f(a)\neq 0$, 
where $W={}^{\alpha}[U/E]^{\bar{p}}$, with $E$ as defined in lemma \ref{diagonal}
and  $I_i$ $(i<\alpha)$ is the concrete interior operator 
defined  in \ref{interior}.
\item  Assume that $\C$ is an $S5$ dimension complemented Chang algebra, $F$ is a Henkin  ultrafilter of $\C$ and $a\in F$. 
Then there exist a Chang system  $V: U/E\to \wp(\wp(U/E))$, $p\in {}^{\alpha}U$, and a homomorphism 
$f:\C\to (\wp(W), \Box_i)_{i<\alpha}$ with $f(a)\neq 0$, 
where $W={}^{\alpha}[U/E]^{(\bar{p})}$ and the concrete box operators 
are defined  from $V$ as in \ref{interior}.
\end{enumarab}
\end{lemma}
\begin{proof}
We prove the first item. The proof of the second item is the same. Let $W={}^{\alpha}[\alpha/E]^{(\bar{Id})}.$
Define, as we did before,  $f:\A\to \wp(W)$ via
$$p\mapsto \{\bar{\tau} \in W: {\sf s}_{\tau}p\in F\}.$$
For $i\in \alpha$ and $p\in \A$, let 
$$O_{p,i}=\{k/E\in \alpha/E: {\sf s}_i^kI(i)p\in F\}.$$
Let $${\cal B}=\{O_{p,i} : i\in \alpha, p\in A\}.$$
Then it is easy to check that ${\cal B}$ 
is the base for a topology on $\alpha/E$.

To define the interior operations, we set
for each $i<\alpha$ 
$$J_i: \wp(W)\to \wp (W)$$
by $$[x]\in J_iX\Longleftrightarrow \exists U\in {\cal B}(x_i/E\in U\subseteq \{u/E\in \alpha/E: [x]^i_{u/E}\in X\}),$$
where $X\subseteq V$. Note that 
$[x]^i_{u/E}=[x^i_u]$.
We now check that $f$ preserves the interior operators $J_i$ $(i<\alpha)$, too.
We need to show 
$$\psi(I_ip)=J_i(\psi(p)).$$
The reasoning is like \cite{g}; the difference is that in \cite{g}, the constants
denoted by  $x_i$ are endomorphisms on $\A$; the value $x_i$ at $j$ 
corresponds in our adopted approach to ${\sf s}^j_u$ where $u=x_i(j)$.
Let $[x]$ be in $\psi(l_ip)$. Let 
$${\sf sup}(x)=\{k\in \alpha: x_k\neq k\}.$$
Then, by definition,  ${\sf s}_xI_ip\in F$. Hence 
$${\sf s}_{x_i}^i I_i{\sf s}^{i_1}_{x_1}\ldots {\sf s}^{i_n}_{x_n} p\in F,$$
where
$${\sf sup}(x)\sim \{i\}=\{j_1,\ldots, j_n\}.$$
Let 
$$y=[j_1|x_1]\circ \ldots [j_n| x_n].$$
Then $x_i/E\in \{u/E: {\sf s}_u^i I(i){\sf s}_yp\in F\}\in q.$
But $I_i{\sf s}_yp\leq {\sf s}_y p,$ hence
$$U=\{u/E: {\sf s}_u^i I_i{\sf s}_yp\in F\}\subseteq \{u/E: {\sf s}_u^i{\sf s}_yp\in F\}.$$
It follows that $x_i/E\in U\subseteq \{u/E: x^i_u\in \Psi(p)\}.$
Thus $[x]\in J_i\psi(p).$

Now we prove the other direction.
Let $[x]\in J_i\Psi(p)$. Let $U\in \B$ be such that 
$$x_i/E\in U\subseteq \{u/E\in \alpha/E: {\sf s}_u^i{\sf s}_xp\in F\}.$$
Assume that $U=O_{r,j}$, where  $r\in \A$ and $j\in \alpha$.
Let $u\in \alpha\sim [\Delta p\cup \Delta r\cup \{i,j\}]$. By dimension complementedness 
such a $u$ exists. 
Then we have:  
\begin{align*}
{\sf s}_u^jI_jr\in F&\Longleftrightarrow {\sf s}_u^i{\sf s}_xp\in F,\\
{\sf s}_u^jI_jr\cdot {\sf s}_u^i{\sf s}_xp\in F&\Longleftrightarrow {\sf s}_u^jI_jr\in F.
\end{align*}
But
${\sf s}_u^jI_jr={\sf s}_u^iI_i{\sf s}_j^ir$, so we have
\begin{align*}
{\sf s}_u^iI_i{\sf s}_j^ir\cdot {\sf s}_u^i{\sf s}_xp &\oplus {\sf s}^i_uI_i{\sf s}_j^ir\in F,\\
{\sf s}_u^i[I_i{\sf s}_j^ir\cdot {\sf s}_xp&\oplus I_i{\sf s}_j^ir]\in F,\\
{\sf q}_i[I_i{\sf s}_j^ir\cdot {\sf s}_xp&\oplus I_i{\sf s}_j^ir]\in F,\\
{\sf q}_i[I_i{\sf s}_j^ir\cdot I_i{\sf s}_xp&\oplus I_i{\sf s}_j^ir]\in F,\\
{\sf s}^i_{x_i}[I_i{\sf s}_j^ir\cdot I_i{\sf s}_xp&\oplus I_i{\sf s}_j^ir]\in F,\\
{\sf s}^j_{x_i}I_jr\cdot s^i_{x_i}I_i{\sf s}_xp&\oplus {\sf s}^j_{x_i}I_jr\in F.\\
\end{align*}
But ${\sf s}^j_{x_i}I_jr\in F,$ hence 
${\sf s}^i_{x_i}I_i{\sf s}_xp\in F$, and so $x\in \Psi(I_ip)$ as required.

For the second part,  define $V: U/E\to \wp(\wp(U/E))$ by
$$V(m/E)=\{\{j/E\in U/E: {\sf s}^i_jp\in F\}: p\in A, i\in I, {\sf s}^i_m\Box(i)p\in F\}.$$ 
Using the above reasoning together with the reasoning in \cite{g2} p. 46-47, it can be checked that
$V$ is as required.
\end{proof}

Having lemma \ref{essence} at hand, we can now show that 
Henkin constructions used for cylindric algebras to prove the interpolation theorems 
\cite{IGPL, AUamal} works
when the algebras are endowed with interior operators. The only significant difference 
between the coming proof and the proofs in the two cited references is that in these references the 
interpolation property was proved for {\it countable} (possibly dimension restricted) free algebras. 
To get round the obstacle
of uncountability, we use dilations to {\it regular cardinals} which gives us `enough space'. 
This guarantees that witnesses can always be found and do not cash with cylindrifiers. 
The reader is referred to \cite{IGPL} for omitted details. 
Recall that for an algebra $\A$ and $X\subseteq \A$, $\Sg^{\A}X$ is the subalgebra of $\A$ generated by $X$.

The algebraic version of the {\it Craig interpolation property} is defined as follows:

\begin{definition} An algebra $\A\in \sf TCA_{\alpha}$ has the {\it interpolation property} 
if for all $X_1, X_2\subseteq \A$, if whenever $a\in \Sg^{\A}X_1$ and $c\in \Sg^{\A}X_2$ are such that 
$a\leq c,$ then there exists $b\in \Sg^{\A}(X_1\cap X_2)$ such that $a\leq b\leq c$, in which case we say that
$b$ is {\it an interpolant of $a$ and $c$} or even  simply {\it an interpolant.}
\end{definition}

\begin{theorem}\label{in} Let $\alpha$ be an infinite ordinal. let $\beta$ be a cardinal.  Let $\rho:\beta\to \wp(\alpha)$ such that
$\alpha\sim \rho(i)$ is infinite for all $i\in \beta$. Then $\Fr_{\beta}^{\rho}\TCA_{\alpha}$ has the interpolation property.
\end{theorem}
\begin{proof}
Let $\A=\Fr_{\beta}^{\rho}\TCA_{\alpha}$. 
Let $a\in \Sg X_1$ and $c\in \Sg X_2$ be such that $a\leq c$. We want to find an interpolant in 
$\Sg^{\A}(X_1\cap X_2)$. By lemma \ref{prep} let $\B\in \sf TCA_{\kappa}$, $\kappa$ a regular cardinal, such that $\A=\Nr_{\alpha}\B$. 
Assume that no such interpolant exists in $\A$, then no interpolant exists in $\B$, because if $b$ is an interpolant
in $\Sg^{\B}(X_1\cap X_2),$ then there exists a finite set $\Gamma\subseteq \kappa\sim \alpha$, such that
${\sf c}_{(\Gamma)}b\in \Nr_{\alpha}\Sg^{\B}(X_1\cap X_2)=\Sg^{\Nr_{\alpha}\B}(X_1\cap X_2)=\Sg^{\A}(X_1\cap X_2)$; 
which is clearly an interpolant in $\A$.
Arrange $\kappa\times \Sg^{\B}X_1$ 
and $\kappa\times \Sg^{\B}X_2$ 
into $\kappa$-termed sequences:
$$\langle (k_i,x_i): i\in \kappa\rangle\text {  and  }\langle (l_i,y_i):i\in \kappa\rangle
\text {  respectively.}$$ Since $\kappa$ is regular, we can define by recursion 
$\kappa$-termed sequences of witnesses: 
$$\langle u_i:i\in \kappa\rangle \text { and }\langle v_i:i\in \kappa\rangle$$ 
such that for all $i\in \kappa$ we have:
$$u_i\in \mu\smallsetminus
(\Delta a\cup \Delta c)\cup \cup_{j\leq i}(\Delta x_j\cup \Delta y_j)\cup \{u_j:j<i\}\cup \{v_j:j<i\}$$
and
$$v_i\in \mu\smallsetminus(\Delta a\cup \Delta c)\cup 
\cup_{j\leq i}(\Delta x_j\cup \Delta y_j)\cup \{u_j:j\leq i\}\cup \{v_j:j<i\}.$$

For a Boolean algebra $\C$  and $Y\subseteq \C$, we write 
$fl^{\C}Y$ to denote the Boolean filter generated by $Y$ in $\C.$
Now let 
\begin{align*}
Y_1&= \{a\}\cup \{-{\sf  c}_{k_i}x_i+{\sf s}_{u_i}^{k_i}x_i: i\in \kappa\},\\
Y_2&=\{-c\}\cup \{-{\sf  c}_{l_i}y_i +{\sf s}_{v_i}^{l_i}y_i:i\in \kappa\},\\
H_1&= fl^{Bl\Sg^{\B}(X_1)}Y_1,\  H_2=fl^{Bl\Sg^{\B}(X_2)}Y_2,\\ 
H&=fl^{Bl\Sg^{\B}(X_1\cap X_2)}[(H_1\cap \Sg^{\B}(X_1\cap X_2)
\cup (H_2\cap \Sg^{\B}(X_1\cap X_2)].
\end{align*}
Then $H$ is a proper filter of $\Sg^{\B}(X_1\cap X_2)$ \cite{IGPL}.
Proving that $H$ is a proper filter of $\Sg^{\B}(X_1\cap X_2)$,
let $H^*$ be a (proper Boolean) ultrafilter of $\Sg^{\B}(X_1\cap X_2)$
containing $H.$
We obtain  ultrafilters $F_1$ and $F_2$ of $\Sg^{\B}X_1$ and $\Sg^{\B}X_2$,
respectively, such that
$$H^*\subseteq F_1,\ \  H^*\subseteq F_2$$
and (**)
$$F_1\cap \Sg^{\B}(X_1\cap X_2)= H^*= F_2\cap \Sg^{\B}(X_1\cap X_2).$$
Now for all $x\in \Sg^{\B}(X_1\cap X_2)$ we have
$$x\in F_1\text { if and only if } x\in F_2.$$
Also from how we defined our ultrafilters, $F_i$ for $i\in \{1,2\}$ are Henkin, 
that is, 
they  satisfy the following condition:

(*) For all $k<\mu$, for all $x\in \Sg^{\B}X_i$
if ${\sf  c}_kx\in F_i$ then ${\sf s}_l^kx$ is in $F_i$ for some $l\notin \Delta x.$
We obtain  ultrafilters $F_1$ and $F_2$ of $\Sg^{\B}X_1$ and $\Sg^{\B}X_2$, 
respectively, such that 
$$H^*\subseteq F_1,\ \  H^*\subseteq F_2$$
and (**)
$$F_1\cap \Sg^{\B}(X_1\cap X_2)= H^*= F_2\cap \Sg^{\B}(X_1\cap X_2).$$
Now for all $x\in \Sg^{\B}(X_1\cap X_2)$ we have 
$$x\in F_1\text { if and only if } x\in F_2.$$ 
  
Fix $m\in \{1,2\}$. the definition of the representations here slightly differs from the definition in \ref{essence} for the equivalence
relation $E$ is now defined on $\beta$ the dilated dimension, but this does not alter the proof that the maps to be defined
using the hitherto constructed Henkin ultrafilters are homomorphisms.  
In more detail, let $V={}^{\alpha}\beta^{(Id)}$. $E$ denotes the equivalence relation on $\beta$ defined via $(i, j)\in E$ iff ${\sf d}_{ij}\in F_m$.
Now define for $\sigma, \tau\in V$,  $\sigma \bar{E}\tau$ 
iff ${\sf d}_{\sigma(i), \tau(i)}\in F_m$ for all $i\in \alpha$. Let $W= V/\bar{E}.$ 
For $h\in W,$ write $h=\bar{\tau}$ for $\tau\in V$ such that
$\tau(j)/E=h(j)$ for all $j\in \alpha$. Define for $i<\alpha$, and 
$X\subseteq W={}^{\alpha}(\beta/\bar{E})^{\bar{Id}}$ the $i$th interior operator
$$I_i(X)=\{s\in W: s_i\in \{u/E\in \beta/E: s^i_u\in X\}\}.$$
Now define, as in lemma \ref{essence}, $f_m:\Sg^{\A}X_m\to (\wp(W), I_i)_{i<\alpha}$ by
$$f_m(a)=\{\bar{\tau}\in W: {\sf s}_{\tau\cup Id}^{\B}a\in F_m\}.$$
 
It can be checked exactly as before that $f_m$ is a homomorphism. 

Without loss of generality, we can assume that $X_1\cup X_2=X.$
We have $f_1$ and $f_2$ agree on $X_1\cap X_2$. 
So that $f_1\cup f_2$ defines a function on $X_1\cup X_2$. 
By dimension restricted freeness,  
it follows that there is a homomorphism  $f$ from $\A$ to $(\wp (W), I_i)_{i<\alpha}$ 
such that $f_1\cup f_2\subseteq f$.   Then $\bar{Id}\in f(a)\cap f(-c) = f(a\cdot -c).$ 
This is so because ${\sf s}_{Id}a=a\in F_1$
${\sf s}_{Id}(-c)=-c\in F_2.$ But this contradicts 
the premise that $a\leq c.$

\end{proof}

The respresentabilty of $\TDc_{\alpha}$s can be discerned below the surface of the previous proof, so that
the representability result in \cite{g} is a special case.  In more detail, we have:

\begin{corollary}\label{representability} Every algebra $\A\in \TDc_{\alpha}$ is representable.
\end{corollary}
\begin{proof} Let $\A$ be given and $a\neq 0$ be in $A$. Let $\kappa$ be a regular cardinal $\geq max\{|\alpha|, |A|\}$.
Let $\B\in \TCA_{\kappa}$ be such that $\A=\Nr_{\alpha}\B$. Let $\langle (k_i,x_i): i\in \kappa \rangle$ be an enumeration of $\kappa\times B.$
Since  $\kappa$ is regular, we can define by recursion a
$\kappa$-termed sequence  $\langle u_i:i\in \kappa\rangle$
such that for all $i\in \kappa$ we have:
$u_i\in \kappa\sim
(\Delta a\cup \bigcup_{j\leq i}\Delta x_j\cup \{u_j:j<i\}).$
Let  $Y=\{a\}\cup \{-{\sf  c}_{k_i}x_i+{\sf s}_{u_i}^{k_i}x_i: i\in \kappa\}.$ Let $H$ be the filter generated by $Y$;
then $H$  is proper, take the maximal filter containing $H$ and $a$,
and define $\psi(b)=\{\bar{\tau}\in W: {\sf s}_{\tau}b\in F\}$ where  $b\in B$ and $W$ is as defined in the previous proof.
Then $\psi(a)\neq 0$, and $\psi$ establishes the 
representability of $\B$, hence of $\A$.
\end{proof}

\subsection{Omitting types}

Now we prove an omitting types theorem for $\TDc_{\alpha}$ and $\sf TLf_{\alpha}$ when $\alpha$ is a countable 
infinite ordinal; also generalizing the result in \cite{g} which addresses only topological locally finite algebras.
An omitting types theorem for Chang modal logic is not proved in \cite{g2}. 

The proof adopted herein, we find is much simpler than the proof
in \cite{g}; and it resorts to the Baire category theorem for compact Hausdorff spaces
as is often the case with `omitting types constructions' though they are rarely presented this way. 

The proof is similar to the proof of \cite[Theorem 3.2.4]{Sayed} having at our disposal 
lemma \ref{essence}.
We omit the parts of the proof that overlap with those in \cite{Sayed}.
But we still need some preparing to do.

Given $\A\in \TCA_{\alpha}$, $X\subseteq \A$ is called a {\it finitary type}, if 
$X\subseteq \Nr_n\A$ for some $n\in \omega$. It is non-principal if $\prod X=0$.

A {\it representation} of   $\A\in \sf TDc_{\alpha}$ is a non-zero  homomorphism 
$f:\A\to \B$ where $\B$ is a weak set algebra. If $\A$ is simple then $f$ is necessarily an isomorphism. 
$X\subseteq \A$ is {\it omitted}
by $f$ if $\bigcap_{x\in X} f(x)=\emptyset$, 
otherwise it is {\it realized} by $f$. 

Let $cov K$ be the least cardinal $\kappa$ such that the real line can be coverd by $\kappa$ no-where dense sets.
$cov K$ is a cardinal closely related to the number of omitting types and to 
independent set theoretic axioms like Martin's axiom restricted to countable Boolean algebras.
It also has topological re-incarnations, 
closely related to the Baire category theorem, witness \cite{Sayed} for a discussion of properties of this cardinal.

Let $\A$ be any Boolean algebra. The set of ultrafilters of $\A$ is denoted 
by $\mathfrak{U}(\A)$. The Stone topology  makes $\mathfrak{U}(\A)$ a compact Hausdorff space. 
We denote this space by $\A^*$. Recall that the Stone topology has as its basic open sets the sets $\{N_x:x\in A\}$ where
$$N_x=\{F\in\mathfrak{U}(\A):x\in F\}.$$
%Note that, if $A$ is countable, then by theorem \ref{polish} $\A^*$ is a Polish space.

Let $x\in A$, $Y\subseteq A$ and suppose that $x=\sum Y.$ 
We say that an ultrafilter $F\in\mathfrak{U}(\A)$ \emph{preserves} 
$Y$ iff $x\in F$ implies that $y\in F$ for some $y\in Y$.

Now let $\A\in \sf TLf_\omega$. For each $i\in\omega$ and $x\in A$ let
$$\mathfrak{U}_{i,x}=\{F\in\mathfrak{U}(\A):F\mbox{ preserves }\{{\sf s}^i_jx:j\in\omega\}\}.$$
Then
\begin{align*}\mathfrak{U}_{i,x}&=\{F\in\mathfrak{U}(\A): {\sf c}_ix\in F\Rightarrow (\exists j\in\omega) {\sf s}^i_jx\in F\}\\
&=N_{-{\sf c}_ix}\cup\bigcup_{j<\omega}N_{{\sf s}^i_jx}.
\end{align*}

Let $$\mathcal{H}(\A)=\bigcap_{i\in\omega,x\in A}\mathfrak{U}_{i,x}(\A)\cap\bigcap_{i\neq j}N_{-{\sf d}_{ij}}.$$
It is clear that $\mathcal{H}(\A)$ is a $G_\delta$ set in $\A^*$.

For $F\in\mathfrak{U}(\A),$ let $$rep_F(x)=\{\tau\in{}^\omega\omega: {\sf s}_\tau^{\A} x\in F\},$$ for all $x\in A.$ 
Here for $\tau\in {}^{\omega}\omega$, ${\sf s}_{\tau}^{\A}x$ by definition is ${\sf s}_{\tau\upharpoonright \Delta x}^{\A}x$. 
The latter is well defined because  $|\Delta x|<\omega.$  

When $a\in F$, then $rep_F$ is a representation of $\A$ such that $rep_F(a)\neq 0$. 
Notice that here we do not have a notion of quotient involved here defined via the diagonal elements. 
Preservation of diagonal elements is guaranteed
by the fact that $-{\sf d}_{ij}\in F$. As before, it is easy to 
check that the cylindrifiers are preserved as well because the ultrafilter is 
Henkin.

The following theorem is due to S\'agi \cite{Sagi}, establishing a one to one correpondance between representations 
of locally finite cylindric algebras and Henkin 
ultrafilters. $\sf Cs_{\omega}^{reg}$ denotes the class of {\it regular set algebras}; a a set algebra with top element
$^{\alpha}U$ is such, if whenever $f, g\in {}^{\alpha}U,$ $f\upharpoonright \Delta x=g\upharpoonright \Delta x$, and  
$f\in X$ then $g\in X$. This reflects
the metalogical property that if two assignments agree on 
the free variables occuring in a formula then both satisfy the
formula or none does.

\begin{theorem}\label{th1}
If $F\in\mathcal{H}(\A)$, then $rep_F$ is a homomorphism from $\A$ onto an element 
of $\sf Lf_\omega\cap \sf Cs_\omega^{reg}$ 
with base $\omega$. Conversely, if $h$ is a homomorphism from $\A$ onto an element of 
$\sf Lf_\omega\cap \sf Cs_\omega^{reg}$ with base $\omega$, then there is a unique $F\in \mathcal{H}(\A)$ such that $h=rep_F.$
\end{theorem}

The next theorem is due to Shelah, and will be used to show that in certain cases uncountably many non-principal types can be 
omitted.
\begin{theorem}\label{Shelah} Suppose that $T$ is a theory,
$|T|=\lambda$, $\lambda$ regular, then there exist models $\M_i: i<{}^{\lambda}2$, each of cardinality $\lambda$,
such that if $i(1)\neq i(2)< \chi$, $\bar{a}_{i(l)}\in M_{i(l)}$, $l=1,2,$, $\tp(\bar{a}_{l(1)})=\tp(\bar{a}_{l(2)})$,
then there are $p_i\subseteq \tp(\bar{a}_{l(i)}),$ $|p_i|<\lambda$ and $p_i\vdash \tp(\bar{a}_ {l(i)})$ 
($\tp(\bar{a})$ denotes the complete type realized by
the tuple $\bar{a}$).
\end{theorem}
\begin{proof} \cite[Theorem 5.16, Chapter IV]{Shelah}.
\end{proof}

We shall use the algebraic counterpart of the following corollary obtained by restricting Shelah's theorem to the countable case:
\begin{corollary}\label{Shelah2} For any countable theory, there is a
family of $< {}^{\omega}2$ countable models that overlap only on principal types.
\end{corollary}

\begin{theorem}\label{infinite} 
\begin{enumarab}
\item Let $\A\in \TDc_{\omega}$ be  countable. Assume that 
$\kappa<covK$. Let $(\Gamma_i:i\in \kappa)$ be a set of non-principal types in $\A$. 
Then there is a topologocal  weak set algebra $(\B, I_i)_{i<\omega}$, that is,  $\B$ has top element a weak space, 
and a  homomorphism  $f:\A\to (\B, I_i)_{i<\omega},$ 
such that for all $i\in \kappa$, $\bigcap_{x\in X_i}f(x)=\emptyset$, and $f(a)\neq 0.$

\item If $\A\in \sf TLf_{\omega}$, and $(\Gamma_i: i\in \kappa)$ 
is a family of finitary non-principal types  
then there is a topogological  set algebra $(\B, I_i)_{i<\omega}$, that is, $\B$ has  top element a 
cartesian square,  and $\B\in {\sf Cs}_{\omega}^{reg}\cap \sf Lf_{\omega}$ together 
with a homomorphism  $f:\A\to (\B, I_i)_{i<\omega}$ 
such that $\bigcap_{x\in X_i}f(x)=\emptyset$, and $f(a)\neq 0.$ 
 
If the types are  maximal  
then $covK$ can be replaced by $2^{\omega}$, so that $<2^{\omega}$ types can be omitted.
\end{enumarab}
\end{theorem}
\begin{proof} 
\begin{enumarab}
\item For the first part, we have by \cite[1.11.6]{HMT1} that 
\begin{equation}\label{t1}
\begin{split} (\forall j<\alpha)(\forall x\in A)({\sf c}_jx=\sum_{i\in \alpha\smallsetminus \Delta x}
{\sf s}_i^jx.)
\end{split}
\end{equation}
Now let $V$ be the weak space $^{\omega}\omega^{(Id)}=\{s\in {}^{\omega}\omega: |\{i\in \omega: s_i\neq i\}|<\omega\}$.
For each $\tau\in V$ for each $i\in \kappa$, let
$$X_{i,\tau}=\{{\sf s}_{\tau}x: x\in X_i\}.$$
Here ${\sf s}_{\tau}$ 
is the unary operation as defined in  \cite[1.11.9]{HMT1}.
For each $\tau\in V,$ ${\sf s}_{\tau}$ is a complete
boolean endomorphism on $\A$ by \cite[1.11.12(iii)]{HMT1}. 
It thus follows that 
\begin{equation}\label{t2}\begin{split}
(\forall\tau\in V)(\forall  i\in \kappa)\prod{}^{\A}X_{i,\tau}=0
\end{split}
\end{equation}
Let $S$ be the Stone space of the Boolean part of $\A$, and for $x\in \A$, let $N_x$ 
denote the clopen set consisting of all
Boolean ultrafilters that contain $x$.
Then from \ref{t1}, \ref{t2}, it follows that for $x\in \A,$ $j<\beta$, $i<\kappa$ and 
$\tau\in V$, the sets 
$$\bold G_{j,x}=N_{{\sf c}_jx}\setminus \bigcup_{i\notin \Delta x} N_{{\sf s}_i^jx}
\text { and } \bold H_{i,\tau}=\bigcap_{x\in X_i} N_{{\sf s}_{\bar{\tau}}x}$$
are closed nowhere dense sets in $S$.
Also each $\bold H_{i,\tau}$ is closed and nowhere 
dense.
Let $$\bold G=\bigcup_{j\in \beta}\bigcup_{x\in B}\bold G_{j,x}
\text { and }\bold H=\bigcup_{i\in \kappa}\bigcup_{\tau\in V}\bold H_{i,\tau.}$$
By properties of $covK$, it can be shown $\bold H$ is a countable collection of nowhere dense sets.
By the Baire Category theorem  for compact Hausdorff spaces, we get that $H(A)=S\sim \bold H\cup \bold G$ is dense in $S$.
Accordingly, let $F$ be an ultrafilter in $N_a\cap X$.
By the very choice of $F$, it follows that $a\in F$ and  we have the following 
\begin{equation}
\begin{split}
(\forall j<\beta)(\forall x\in B)({\sf c}_jx\in F\implies
(\exists j\notin \Delta x){\sf s}_j^ix\in F.)
\end{split}
\end{equation}
and 
\begin{equation}
\begin{split}
(\forall i<\kappa)(\forall \tau\in V)(\exists x\in X_i){\sf s}_{\tau}x\notin F. 
\end{split}
\end{equation}
Let $V={}^{\omega}\omega^{Id)}$ and let 
$W$ be the quotient of $V$ as defined above.
That is $W=V/\bar{E}$ where $\tau\bar{E}\sigma$ if ${\sf d}_{\tau(i), \sigma(i)}\in F$ 
for all $i\in \omega$.

Define $f$ as before by $$f(x)=\{\bar{\tau}\in W:  {\sf s}_{\tau}x\in F\}, \text { for } x\in \A.$$ 
and the interior operators  
for each $i<\alpha$ by
$$J_i: \wp(W)\to \wp (W)$$
by $$[x]\in J_iX\Longleftrightarrow \exists U\in {\cal B}(x_i/E\in U\subseteq \{u/E\in \alpha/E:[x]^i_{u/E}\in X\}),$$
where $X\subseteq W$; here $W$ and $E$ are as defined in lemmas, \ref{diagonal}, \ref{t} and $\cal B$
is the base for the topology on $U/E$ defined as in the proof of theorem \ref{essence}. 
Then by lemma \ref{essence} $f$
is a homomorphism 
such that $f(a)\neq 0$ and it can be easily checked  that  
$\bigcap f(X_i)=\emptyset$ for all $i\in \kappa$, 
hence the desired conclusion.

\item One proceeds exactly like in the previous item, but using, as indicated above, the fact that 
the operations ${\sf s}_{\tau}$ for {\it any} $\tau\in {}^\omega\omega$
which are definable in locally finite algebras, via ${\sf s}_{\tau}x= {\sf s}_{\tau\upharpoonright \Delta x}x$, for
any $x\in A$. Furthermore, ${\sf s}_{\tau}\upharpoonright \Nr_n\A$ is a complete
Boolean endomorphism, so that we guarantee that infimums are preserved and the sets 
$\bold H_{i,\tau}=\bigcap_{x\in X_i} N_{{\sf s}_{\bar{\tau}}x}$
remain no-where dense in the Stone topology. 

Now for the second part. Let  $\A\in \sf TLf_{\alpha}$, $\lambda <2^{\omega}$ and $\bold F=(X_i:i<\lambda)$ be a family 
of maximal non-principal finitary types, so that for each $i<\lambda$, there exists $n\in \omega$ such that 
$X_i\subseteq \Nr_n\A$, and $\prod X_i=0$; that is $X_i$ is a Boolean ultrafilter in $\Nr_n\A$.  
Then by theorem \ref{Shelah}, or rather its direct algebraic counterpart, 
there are $^{\omega}2$ representations such that 
if $X$ is an ultrafilter in $\Nr_n\A$ (some $n\in \omega)$) that is realized in two such 
representations, then $X$  necessarily principal. 
That is there exist a family of countable locally finite set algebras, each  with countable base, call it
$(\B_{j_i}: i<{}2^{\omega})$, and isomorphisms $f_i:\A\to \B_{j_i}$  
such that if $X$ is an ultrafilter in $\Nr_n\A$, for which there exists distinct $k, l\in {}2^{\omega}$
with $\bigcap f_l(X)\neq \emptyset$ and $\bigcap f_j(X)\neq \emptyset$, 
then $X$ is principal, so that  from corollary \ref{Shelah2} such representations overlap only
on maximal principal  types. By theorem \ref{th1}, there exists a family $(F_i: i< 2^{\omega})$
of Henkin ultrafilters 
such that $f_i=h_{F_i}$, and by theorem \ref{essence} we can assume that 
$h_{F_i}$ is a $\TCA_{\alpha}$ isomorphism as follows. Denote $F_i$ by $G$.
For $p\in \A$ and $i<\alpha$, let 
$O_{p,i}=\{k\in \alpha: {\sf s}_i^kI(i)p\in G\}$ 
and let ${\cal B}=\{O_{p,i} : i\in \alpha, p\in A\}.$
Then ${\cal B}$ 
is the base for a topology on $\alpha$ and the concrete interior operations are defined 
for each $i<\alpha$ via
$J_i: \wp(^{\alpha}\alpha)\to \wp (^{\alpha}\alpha)$
$$x\in J_iX\Longleftrightarrow \exists U\in {\cal B}(x_i\in U\subseteq \{u\in \alpha: x^i_{u}\in X\}),$$
where $X\subseteq W$. 

Assume, for contradiction,  that there is no representation (model) 
that omits $\bold F$. 
Then for all $i<{}2^{\omega}$,
there exists $F$ such that $F$ is realized in $\B_{j_i}$. Let $\psi:{}2^{\omega}\to \wp(\bold F)$, be defined by
$\psi(i)=\{F: F \text { is realized in  }\B_{j_i}\}$.  Then for all $i<2^{\omega}$, $\psi(i)\neq \emptyset$.
Furthermore, for $i\neq k$, $\psi(i)\cap \psi(k)=\emptyset,$ for if $F\in \psi(i)\cap \psi(k)$ then it will be realized in
$\B_{j_i}$ and $\B_{j_k}$, and so it will be principal.  This implies that 
$|\bold F|=2^{\omega}$ which is impossible.

\end{enumarab}

\end{proof}

%\section{Two more algebraic results for locally finite and dimension complemented algebras}

From the omitting types theorem proved in theorem \ref{infinite}(2), 
one can infer the existence of an {\it atomic representation} of a {\it neatly atomic} $\sf TLf_{\omega}.$
In the coming two theorems $\A$ will be countable and simple, and we refer 
to an isomorphism $f:\A\to \B$, $\B$ a set algebra as 
a representation of $\A$.

\begin{definition}
\begin{enumarab}
\item $\A\in \sf TLf_{\omega}$ is {\it neatly atomic} if $\Nr_n\A$ is atomic for every $n\in \omega$.
\item A representation $f:\A\to \B$ with base $U$ is an {\it atomic representation} of $\A$ 
if $\bigcup \{f(x): x\in \At\Nr_n\A\}={}^{\omega}U$ for every
$n\in \omega$. 
\end{enumarab}
\end{definition}
\begin{theorem}\label{atomic0}
If $\A$ is a simple neatly atomic countable $\sf TLf_{\omega},$
then $\A$ has an atomic representation $f:\A\to \B$. 
Furthermore, if $(Y_i: i\in I)$ is a family of non-principal finitary types then 
$\bigcap_{x\in Y_i} f(x)=\emptyset.$
\end{theorem}
\begin{proof} The first part follows from the omitting types theorem by taking $X_n$ to 
be the set of co-atoms of $\Nr_n\A$ and then finding
a representation that omits those.
For the second part. 
Let $i\in I$. Assume that $X_i\subseteq \Nr_n\A$. Let $Z_i=\{-y: y\in Y_i\}$. Then $\sum Z_i=1$. So for any atom 
$x\in \Nr_n\A$, we have $x\cdot \sum Z_i=x\neq 0$.
Hence there exists $z\in Z_i$, such that
$x\cdot z\neq 0$. But $x$ is an atom, hence $x\cdot z=x$ and so $x\leq z$. 
We have shown that for every atom $x\in \Nr_n\A$, there exists $z\in Z_i$ such that $x\leq z$.
It follows immediately,  that
$^{\omega}U=\bigcup{f(x): x\in \At\Nr_n\A}\}\leq \bigcup_{z\in Z_i} f(z)$, 
and so, $\bigcap_{y\in Y_i} f(y)=\emptyset,$
and we are done.
\end{proof}
\section{Notions of Representability}

$\sf TCs_{\alpha}$ denotes the class of set algebras and ${\sf TWs}_{\alpha}$ 
denotes the class of weak set algebra. 
Recall that $\A\in \sf TCs_{\alpha}$ if it has top element $^{\alpha}U$, $U$ carries a topology
and the interior operators are defined as in definition \ref{interior},  
while $\A$ is in $\sf TWs_{\alpha}$ if $\A$ has unit a weak space $^{\alpha}U^{(p)}$ and
the interior operator also defined as in definition \ref{interior}.  For $\alpha<\omega$, 
$\sf TCs_{\alpha}=\sf TWs_{\alpha}.$ 

We choose to define the class of representable algebras as follows 
(we will see that there are other possible equivalent 
definitions when $\alpha$ is infinite, namely, to take set algebras with 
square units. For the finite dimensional case this is obviously equivalent).

\begin{definition}
$\A\in \TCA_{\alpha}$ is {\it representable} if it is isomorphic to a
subdirect product of weak set algebras of dimension $\alpha$.
\end{definition}
In the next theorem to allow uniform treatment of the finite and infinite dimensional case, we always consider {\it weak set algebras}, which is
the same as set algebras for the finite dimension case. In this case we have for any $p\in {}^{\alpha}U$, $^{\alpha}U={}^{\alpha}U^{(p)}$. 
We also write for any ordinal $\alpha$, $\alpha+\omega$ which is just
$\omega$ when $\alpha$ is finite.

\begin{theorem}\label{neat}
\begin{enumarab}
\item For $\alpha\geq \omega$ if $\A\in \sf TDc_{\alpha}$ and $\Rd_{ca}\A$ is representable, 
then $\A$ is representable, too.  
\item For any ordinal $\alpha$, $\sf RTCA_{\alpha}={\bf S}\Nr_{\alpha}\TCA_{\alpha+\omega}.$
\item For any pair of infinite ordinals $\alpha<\beta$, ${\bf S}\Nr_{\alpha}\TCA_{\beta}$ is a variety.
In particular, $\sf RTCA_{\alpha}$ is a variety. Furthermore, $\sf RTCA_{\alpha}=\bigcap_{k\in \omega} S\Nr_{\alpha}\TCA_{\alpha+k}.$
\item $\sf RTCA_{\alpha}={\bf SP}{\sf TWs}_{\alpha}$.
In particular, ${\bf SP}{\sf TWs}_{\alpha}$ is closed under $\bold H$.
\item ${\bf HSP}{\sf TCs}_{\alpha}\subseteq \sf RTCA_{\alpha}$.
\item For finite $\alpha$, $\sf TRCA_{\alpha}$ is a discriminator variety, that is not completely aditive, hence is not conjugated.
\item Assume that $\alpha$ is an infinite ordinal. Then for any  class $\sf K$, such that ${\sf Lf}_{\alpha}\subseteq \sf K\subseteq 
\sf RTCA_{\alpha}$, 
we have ${\bf SUp}\sf K={\sf RTCA}_{\alpha}.$ In particular, 
${\bf HSP}{\sf K}={\sf RTCA}_{\alpha}.$
\item ${\sf RTCA}_{\alpha}={\bf HSP}{\sf TCs}_{\alpha}.$
\item For $\alpha>2,$ $\sf RTCA_{\alpha}$ cannot be axiomatized by a set of universal formulas contaning only finitely many variables.
\item For any pair of ordinals $1<\alpha<\beta$ the class of neat reducts $\Nr_{\alpha}\sf TCA_{\beta}$ 
is not elementary.
\end{enumarab}
\end{theorem}
\begin{proof} For the first item.  Assume for simplicity that $\A$ is simple (has no proper congruences) and 
that $h:\Rd_{ca}\A\to \B$ is an isomorphism into a weak set
algebra $\B$. The general case follows easily from this special case. 
Then theorem \ref{th1} provides a Henkin ultrafilter $F$ such that $h=rep_F.$
The interior operators are then represented  as in the proof of theorem \ref{essence}. The same $h$ establishes the required 
isomorphism.
For the other items the proofs for the $\sf CA$ case lift without much difficulty \cite[Theorems 2.6.32, 2.6.35, 2.6.52]{HMT1} and
\cite{IGPL, Andreka} 
taking into account lemma \ref{box}. We give a sample. That for any pair of ordinals $\alpha<\beta$, 
$S\Nr_{\alpha}\sf TCA_{\beta}$ is a variety is exactly like the $\CA$ case. 
To show that ${\sf RTCA}_{\alpha}\subseteq S\Nr_{\alpha}\TCA_{\alpha+\omega}$,
it suffices to consider algebras in ${\sf TWs}_{\alpha}$, 
since $S\Nr_{\alpha}\TCA_{\alpha+\omega}$ is closed
under ${\bf SP}$.  Let  $\A\in \sf TWs_{\alpha}$ and assume that $\A$ has top element 
$^{\alpha}U^{(p)}$. Let $\beta=\alpha+\omega$ and let $p^*\in {}^{\beta}U$ be a fixed sequence
such that $p^*\upharpoonright \alpha=p$. Let $\C$ be the $\TCA_{\beta}$ with top element
$^{\beta}U^{(p^*)}$; cylindrifiers and diagonal elements  are defined the usual way and the interior operators induced
by the topology on $U$. Define $\psi: \A\to \C$
via $$X\mapsto \{s\in {}^{\beta}U^{(p^*)}: s\upharpoonright \alpha\in X\}.$$
Then by lemma \ref{box}, $\psi$ is a homomorphism, further it is injective, and as easily
checked, $\psi$ is a neat embedding that is  
$\psi(\A)\subseteq \Nr_{\alpha}\C$.
Maybe the hardest part is to show that if $\A\in S\Nr_{\alpha}\sf TCA_{\alpha+\omega}$
then it is representable. But this follows from the fact that we can assume that $\A\subseteq \Nr_{\alpha}\B$, where
$\B\in \Dc_{\alpha+\omega}$, and then using theorem \ref{representability} baring in mind that a neat reduct of a representable algebra is 
representable. 

For item (7), let $k$ be finite$>1$. Take $\A$ obtained 
by splitting an atom in a set algebra into $k+1$ atoms as done in \cite{Andreka}, expanded by interior operators defined
as identity operator. $\A$ will have a non representable cylindric reduct but its $k$ generated subalgebras will 
be representable. Expand such representations  by the identity functions
and stimulate their representation using the 
discrete topology on the base. 

Finally for non elementarity, in \cite{IGPL}, for any ordinal $\alpha>1$ two weak cylindric set set algebras having top element $V$ are constructed
such that $\B\subseteq \A$, $\B\notin \Nr_{\alpha}\CA_{\alpha+1}$, $\B\equiv \A$  and $\A\in \Nr_{\alpha}\CA_{\alpha+\omega}$.
Give the base of $\B$ the discrete topology, then clearly the resulting expanded structure, call it $\B^*$, by the induced (identity) 
interior operators is still not in 
$\Nr_{\alpha}\TCA_{\alpha+1}$. Now we want to expand the algebra $\A$ to an algebra in $\sf TCA_{\alpha}$
in such a way to preserve elementary equivalence, so we do not have a choice but to
expand it with the identity interior operations, call the resulting algebra $\A^*$. But we also want an $\alpha+\omega$ 
dimensional topological cylindric algebra such that
$\A^*$ neatly embeds into $\D$ and {\it exhausts} its $\alpha$ dimensional element, that is the neat embeding is
onto $\Nr_{\alpha}\D$. As above we can assume that $\A=\Nr_{\alpha}\C$, $\C\in \sf TDc_{\alpha+\omega}$. Give
the base of any representation $\D$ say, of $\C$ the discrete topology forming $\D^*$. 
Then $\A^*=\Nr_{\alpha}\D^*$, $\B^*\notin \Nr_{\alpha}\sf TCA_{\alpha+1}$  
and $\A^*\equiv \B^*$. 
\end{proof}

\subsection{Rainbows, atom-canonicity}

In this section we deal only with finite dimensional algebras.
Throughout, unless otherwise explicity indicated, $n$ wil be a finite ordinal $>2$. 

Notions like {\it \d\ completions, and atom-canonicity} \cite{HHbook} 
are problematic for ${\sf TRCA}_{\alpha}$ 
because the interior operators are {\it not} completely additive. 
However, in some cases the interior operators  can turn out  completey additive (e.g when they are equal to the identity map). 
If $\A$ is such an  {\it atomic} algebra, that is an algebra whose interior operators are completely additive,
then the \d\ completion exists and it is the {\it complex algebra of its atom structure}, in symbols $\Cm\At\A$.
This prompts the following definition:

\begin{definition} A variety $V$ of Boolean algebras with operators is {\it atom-canonial}, 
if for every completely additive atomic $\A\in V$, its \d\ completion, namely,
$\Cm\At\A$ is also in $V.$
\end{definition}

So using constructions for cylindric algebras 
one can construct such a representable algebra whose \d\ completion is {\it not} representable, a task done by Hodkinson for cylindric algebras
\cite{Hodkinson}, but now we considerably sharpen Hodkinson's result by passing to the $\Sc$  reducts of certain topological cylindric algebras 
to be constructed.  For $\A=(A, +, \cdot, -, {\sf c}_i, {\sf d}_{ij}, I_i)_{i,j\in \alpha}\in \TCA_{\alpha}$ its $\Sc$ reduct denoted by $\Rd_{sc}\A$
is the algebra $\A=(A, +, \cdot , -, {\sf c}_i, {\sf s}_i^j)_{i,j\in \alpha}$ where ${\sf s}_i^jx={\sf c}_j(x \cdot {\sf d}_{ij})$ for $i\neq j$
and ${\sf s}_i^ix=x$. Such algebras are called
{\it Pinter's substitution algebras}; they are 
also diagonal-free  
reducts of $\sf CA$s. Here reducts are taken in the generalized sense, so that the opeartions of the reduct are {\it term definable} in the 
expansion.
 
We emphasize that the next result {\it cannot} be obtained by lifting the relation algebra case  \cite[ lemmas 17.32, 17.34, 17.35, 17.36]{HHbook}
to cylindric algebras 
using  Hodkinson's construction in \cite{AU}. Hodkinson constructs from 
every atomic relation algebra an atomic cylindric algebra of dimension $n$, for any $n\geq 3$, 
but the relation algebras {\it does not} embed into the $\sf Ra$ reduct of the constructed
cylindric algebra when $n\geq 6$. If it did, then the $\sf RA$ result would lift as indeed is the case with 
$n=3$. We instead start from scratch. We use a  rainbow cylindric algebra.

In \cite{HHbook2} the rainbow cylindric algebra of dimension $n$ on a graph $\Gamma$ is denoted by $\R(\Gamma)$.
We consider $R(\Gamma)$ to be in $\sf TCA_n$ be expanding its signature with $n$ operators
each interpreted as the identity map. In what follows we consider $\Gamma$ to be the indices of the reds, and for a complete irreflexive graph
$\G$, by $\TCA_{\G, \Gamma}$
we mean the rainbow topological cylindric algebra $\R(\Gamma)$  of dimension $n$,
where ${\sf G}=\{\g_i: 1\leq i<n-1\}\cup \{\g_0^i: i\in \G\}$.

More generally, we consider a rainbow topological cylindric algebra based on relational structures 
$A, B$, to be the rainbow algebra
with signature the binary
colours (binary relation symbols)
$\{\r_{ij}: i,j\in B\}\cup \{\w_i: i<n-1\}\cup \{\g_i:1\leq i<n-1\}\cup \{\g_0^i : i\in A\}$ and $n-1$
shades of yellow ($n-1$ ary relation symbols)
$\{\y_S: S\subseteq_{\omega} A, \text { or } S=A\}.$

We look at models of the rainbow theorem as coloured graphs \cite{HH}.
This class is denoted by ${\sf CRG}.$

A coloured graph is a graph such that each of its edges is labelled by the colours in the above first three items,
greens, whites or reds, and some $n-1$ hyperedges are also
labelled by the shades of yellow.
Certain coloured graphs will deserve special attention.

\begin{definition}
Let $i\in A$, and let $M$ be a coloured graph  consisting of $n$ nodes
$x_0,\ldots,  x_{n-2}, z$. We call $M$ an {\it $i$ - cone} if $M(x_0, z)=\g_0^i$
and for every $1\leq j\leq n-2$, $M(x_j, z)=\g_j$,
and no other edge of $M$
is coloured green.
$(x_0,\ldots, x_{n-2})$
is called {\it the center of the cone}, $z$ {\it the apex of the cone}
and {\it $i$ the tint of the cone.}
\end{definition}

The class of coloured graphs $\sf CRG$ are

\begin{itemize}

\item $M$ is a complete graph.

\item $M$ contains no triangles (called forbidden triples)
of the following types:
\vspace{-.2in}
\begin{eqnarray}
&&\nonumber\\
%(1', x, y)&&\mbox{unless }x=y\label{forb:id}\\
(\g, \g^{'}, \g^{*}), (\g_i, \g_{i}, \w_i),
&&\mbox{any }1\leq i< n-1\;  \\
%(\g^j_0, \y, \w_f)&&\mbox{unless }f\in P, i\in dom(f)\\
(\g^j_0, \g^k_0, \w_0)&&\mbox{ any } j, k\in A\\
%\label{forb:pim}(\g^i_0, \g^j_0, \r_{kl})&&\mbox{unless } \set{(i, k), (j, l)}\mbox{ is an order-}\\
%&&\mbox{ preserving partial function }A\to B\nonumber\\
%\label{forb:pim2}(\g_i, \g_j, \r_{kl})&&\mbox{if } i=j\mbox{ but }k\neq l\\
%\label{forb: black}(\y,\y,\y), (\y,\y,\bb)\\
\label{forb:match}(\r_{ij}, \r_{j'k'}, \r_{i^*k^*})&& i,j,j',k',i^*, k^*\in B,\\ \mbox{unless }i=i^*,\; j=j'\mbox{ and }k'=k^*
\end{eqnarray}
and no other triple of atoms is forbidden.

\item If $a_0,\ldots,   a_{n-2}\in M$ are distinct, and no edge $(a_i, a_j)$ $i<j<n$
is coloured green, then the sequence $(a_0, \ldots, a_{n-2})$
is coloured a unique shade of yellow.
No other $(n-1)$ tuples are coloured shades of yellow.

\item If $D=\set{d_0,\ldots,  d_{n-2}, \delta}\subseteq M$ and
$M\upharpoonright D$ is an $i$ cone with apex $\delta$, inducing the order
$d_0,\ldots,  d_{n-2}$ on its base, and the tuple
$(d_0,\ldots, d_{n-2})$ is coloured by a unique shade
$\y_S$ then $i\in S.$
\end{itemize}

One then can define a polyadic equality atom structure
of dimension $n$ from the class $\sf CRG$. It is a {\it rainbow atom structure}.  Rainbow atom structures  are what Hirsch and Hodkinson call
atom structures built from a class of models \cite{HHbook2}.
Our models are, according to the original more traditional view \cite{HH}
coloured graphs. So let $\sf CRG$ be the class of coloured graphs as defined above.
Let $$\At=\{a:n \to M, M\in \sf CRG: \text { $a$ is surjective }\}.$$
We write $M_a$ for the element of $\At$ for which
$a:n\to M$ is a surjection.
Let $a, b\in \At$ define the
following equivalence relation: $a \sim b$ if and only if
\begin{itemize}
\item $a(i)=a(j)\Longleftrightarrow b(i)=b(j),$

\item $M_a(a(i), a(j))=M_b(b(i), b(j))$ whenever defined,

\item $M_a(a(k_0),\dots, a(k_{n-2}))=M_b(b(k_0),\ldots, b(k_{n-2}))$ whenever
defined.
\end{itemize}
Let $\At$ be the set of equivalences classes. Then define
$$[a]\in E_{ij} \text { iff } a(i)=a(j).$$
$$[a]T_i[b] \text { iff }a\upharpoonright n\smallsetminus \{i\}=b\upharpoonright n\smallsetminus \{i\}.$$
%Define accessibility relations corresponding to the polyadic (transpositions) operations as follows:
%$$[a]S_{ij}[b] \text { iff } a\circ [i,j]=b.$$
This, as easily checked, defines a $\sf CA_n$
atom structure. The complex algebra of this atom structure is denote by ${\sf CA}_{A, B}$ where $A$ is the greens and 
$B$ is the reds.
For interior operators define
$$[a]I_i [b]\text { iff }a\sim b;$$
this defines an atom structure of  a $\TCA_n$, 
we denote the resulting complex algebra  $\Cm\At$ 
by ${\sf TCA}_{A,B}$.

Consider the following two games on coloured graphs, each with $\omega$ rounds, and limited number of pebbles
$m>n$. They are translations of $\omega$ atomic games played on atomic networks
of a rainbow algebra using a limited number of nodes $m$.
Both games offer \pa\ only one move, namely, a cylindrifier move.

From the graph game perspective both games \cite[p.27-29]{HH} build a nested sequence $M_0\subseteq M_1\subseteq \ldots $.
of coloured graphs.

First game $G^m$.
\pa\ picks a graph $M_0\in \sf CRG$ with $M_0\subseteq m$ and
$\exists$ makes no response
to this move. In a subsequent round, let the last graph built be $M_i$.
\pa\ picks
\begin{itemize}
\item a graph $\Phi\in \sf CRG$ with $|\Phi|=n,$
\item a single node $k\in \Phi,$
\item a coloured graph embedding $\theta:\Phi\smallsetminus \{k\}\to M_i.$
Let $F=\phi\smallsetminus \{k\}$. Then $F$ is called a face.
\pe\ must respond by amalgamating
$M_i$ and $\Phi$ with the embedding $\theta$. In other words she has to define a
graph $M_{i+1}\in C$ and embeddings $\lambda:M_i\to M_{i+1}$
$\mu:\phi \to M_{i+1}$, such that $\lambda\circ \theta=\mu\upharpoonright F.$
\end{itemize}
$F^m$ is like $G^m$, but \pa\ is allowed to resuse nodes.

$F^m$ has an equivalent formulation on atomic networks of atomic algebras.

Let $\delta$ be a map. Then $\delta[i\to d]$ is defined as follows. $\delta[i\to d](x)=\delta(x)$
if $x\neq i$ and $\delta[i\to d](i)=d$. We write $\delta_i^j$ for $\delta[i\to \delta_j]$.

\begin{definition}
Let $2< n<\omega.$ Let $\C$ be an atomic ${\sf CA}_{n}$.
An \emph{atomic  network} over $\C$ is a map
$$N: {}^{n}\Delta\to \At\C,$$
where $\Delta$ is a non-empty set called a set of nodes,
such that the following hold for each $i,j<n$, $\delta\in {}^{n}\Delta$
and $d\in \Delta$:
\begin{itemize}
\item $N(\delta^i_j)\leq {\sf d}_{ij}$
\item $N(\delta[i\to d])\leq {\sf c}_iN(\delta)$
%\item $N(\bar{x}\circ [i,j])= {\sf s}_{[i,j]}N(\bar{x})$ for all $i,j<n$.

\end{itemize}
\end{definition}

\begin{definition}\label{def:games}
Let $2\leq n<\omega$. For any ${\sf Sc}_n$
atom structure $\alpha$ and $n<m\leq
\omega$, we define a two-player game $F^m(\alpha)$,
each with $\omega$ rounds.

Let $m\leq \omega$.
In a play of $F^m(\alpha)$ the two players construct a sequence of
networks $N_0, N_1,\ldots$ where $\nodes(N_i)$ is a finite subset of
$m=\set{j:j<m}$, for each $i$.

In the initial round of this game \pa\
picks any atom $a\in\alpha$ and \pe\ must play a finite network $N_0$ with
$\nodes(N_0)\subseteq  m$,
such that $N_0(\bar{d}) = a$
for some $\bar{d}\in{}^{n}\nodes(N_0)$.

In a subsequent round of a play of $F^m(\alpha)$, \pa\ can pick a
previously played network $N$ an index $l<n$, a {\it face}
$F=\langle f_0,\ldots, f_{n-2} \rangle \in{}^{n-2}\nodes(N),\; k\in
m\sim\set{f_0,\ldots, f_{n-2}}$, and an atom $b\in\alpha$ such that
$$b\leq {\sf c}_lN(f_0,\ldots, f_i, x,\ldots, f_{n-2}).$$
The choice of $x$ here is arbitrary,
as the second part of the definition of an atomic network together with the fact
that $\cyl i(\cyl i x)=\cyl ix$ ensures that the right hand side does not depend on $x$.

This move is called a \emph{cylindrifier move} and is denoted
$$(N, \langle f_0, \ldots, f_{n-2}\rangle, k, b, l)$$
or simply by $(N, F,k, b, l)$.
In order to make a legal response, \pe\ must play a
network $M\supseteq N$ such that
$M(f_0,\ldots, f_{i-1}, k, f_{i+1},\ldots f_{n-2}))=b$
and $\nodes(M)=\nodes(N)\cup\set k$.

\pe\ wins $F^m(\alpha)$ if she responds with a legal move in each of the
$\omega$ rounds.  If she fails to make a legal response in any
round then \pa\ wins.
\end{definition}
In what follows by $S_c\sf K$, where $\sf K$ is a class having a Boolean reduct, 
we understand the class of complete subalgebras of $\sf K$, that is $\A\in S_c\sf K$ if there exists
$\B\in \sf K$ such that $\A\subseteq \B$ and for all $X\subseteq \A$ whenever $\sum^{\A}X=1,$ 
then $\sum^{\B}X=1$.
\begin{theorem}\label{thm:n}
Let $\sf K$ be any class between $\sf Sc$ and $\sf CA$. Let $n<m$, and let $\A$ be an atomic $\sf K_n.$
If $\A\in S_c\Nr_{n}\sf K_m, $
then \pe\ has a \ws\ in $F^m(\At\A).$
\end{theorem}
\begin{proof}\cite[Theorem 33]{r}. Strictly speaking this theorem is proved for relation algebras, but the proof easily lifts to
the $\sf CA$ case.
\end{proof}

\begin{theorem}\label{can} For any finite $n>2$, any class
$\K$ between $S\Nr_n\TCA_{n+3}$ and $\sf TRCA_n$ is not atom-canonical.
\end{theorem}
\begin{proof}
We blow up and blur a finite rainbow cylindric algebra namely $R(\Gamma)$
where $\Gamma$ is the complete irreflexive graph $n+1$, and the greens
are  ${\sf G}=\{\g_i:1\leq i<n-1\}
\cup \{\g_0^{i}: 1\leq i\leq n+1\},$ we denote this finite algebra endowed by the $n$ identity 
interior operators by $\TCA_{n+1, n}.$

Let $\At$ be the rainbow atom structure similar to that in \cite{Hodkinson} except that we have $n+1$ greens and
only $n$ indices for reds, so that the rainbow signature now consists of $\g_i: 1\leq i<n-1$, $\g_0^i: 1\leq i\leq n+1$,
$\w_i: i<n-1$,  $\r_{kl}^t: k<l\in n$, $t\in \omega$,
binary relations and $\y_S$, $S\subseteq n+1$,
$n-1$ ary relations. 

We also have a shade of red $\rho$; the latter is a binary relation but is {\it outside the rainbow signature},
though  it is used to label coloured graphs during a certain game devised to prove representability 
of the term algebra \cite{Hodkinson}, and in fact \pe\ can win the $\omega$ rounded game
and build the $n$ homogeneous model $M$ by using $\rho$ whenever
she is forced a red, as will be shown in a while. 

So $\At$ is obtained from the rainbow atom structure of the algebra $\A$ defined in \cite[section 4.2 starting p. 25]{Hodkinson}  
truncating the greens to be finite (exactly $n+1$ greens). In \cite{Hodkinson} it shown that the complex algebra $\Cm\At\A$ is not representable; 
the result obtained now, because the greens are finite but still outfit the red, is sharper; 
it will imply that
$\Cm\At\notin S\Nr_n\TCA_{n+3}$.

The logics $L^n, L^n_{\infty \omega}$ are taken in the rainbow
signature (without $\rho$).

Now $\Tm\At\in \sf TRCA_n$; this can be proved like  in \cite{Hodkinson}.
Strictly speaking the cylindric reduct of $\Tm\At$ can be proved representable like in 
\cite{Hodkinson}; giving, as usual, the base of the representation the discrete topology we get representability of the 
interior operators as well.
The colours used for coloured graphs involved in building  the finite atom structure of the algebra $\sf TCA_{n+1, n}$ are:
\begin{itemize}

\item greens: $\g_i$ ($1\leq i\leq n-2)$, $\g_0^i$, $1\leq i\leq n+1$,

\item whites : $\w_i: i\leq n-2,$

\item reds:  $\r_{ij}$ $i<j\in n$,

\item shades of yellow : $\y_S: S\subseteq n+2$.

\end{itemize}
with {\it forbidden triples}
\vspace{-.2in}
\begin{eqnarray*}
&&\nonumber\\
%(1', x, y)&&\mbox{unless }x=y\label{forb:id}\\
(\g, \g^{'}, \g^{*}), (\g_i, \g_{i}, \w_i),
&&\mbox{any }1\leq i\leq  n-2  \\
(\g^j_0, \g^k_0, \w_0)&&\mbox{ any } 1\leq j, k\leq n+1\\
\label{forb:match}(\r_{ij}, \r_{j'k'}, \r_{i^*k^*}), &&i,j,i', k', i^*, j^*\in n,\\ \mbox{unless }i=i^*,\; j=j'\mbox{ and }k'=k^*.
\end{eqnarray*}
and no other triple is forbidden.

A coloured graph  is red
if at least one of its edges is labelled red.  
For brevity write $\r$ for $\r_{jk}$($j<k<n$).
If $\Gamma$ is a coloured graph using the colours in $\At \TCA_{n+1, n}$, and $a:n\to \Gamma$ is in $\At\TCA_{n+1,n}$,
then $a':n\to \Gamma'$ with $\Gamma'\in \sf CGR$
is a {\it copy} of $a:n\to \Gamma$ if $|\Gamma|=|\Gamma'|$,  all non red edges and $n-1$ tuples have the same colour (whenever defined) 
and  for all $i<j<n$, for every red $\r$, if  $(a(i), a(j))\in \r$,
then there exits $l\in \omega$ such that $(a'(i), a'(j))\in \r^l$. Here we implicitly require that for distinct $i,j,k<n$, if 
$(a(i),a(j))\in \r$, $(a(j), a(k))\in \r'$, $(a(i), a(k))\in \r''$, and 
$(a'(i), a'(j))\in \r^l_1$, $(a'(j), a'(k))\in [\r']^{l_2}$ and $(a'(i), a'(k))\in [\r'']^{l_3}$, then $l_1=l_2=l_3=l$, say,
so that $(\r^l, [\r']^l, [\r'']^l)$ 
is a consistent triangle in $\Gamma'$.
If $a':n\to \Gamma'$ and $\Gamma'$ is a red graph 
using the colours of the rainbow signature of $\At$, whose reds are $\{\r_{kj}^l: k<j<n, l\in \omega\},$
then there is a unique $a: n\to \Gamma$, $\Gamma$ 
a red graph using the red colours in the rainbow signature of $\sf TCA_{n+1,n}$, namely, $\{\r_{kj}: k<j< n\}$
such that $a'$ is a copy of $a$.
We denote $a$ by $o(a')$, $o$ short for {\it original}; $a$ is the original of its copy $a'$.

For $i<n$, let $T_i$ be the accessibility relation corresponding to the $i$th cylindrifier in $\At$. 
Let  
$T_i^{s}$, be that corresponding to the $i$th cylindrifier in ${\sf TCA}_{n+1, n}$. 
Then if $c:n\to \Gamma$ and $d: n\to \Gamma'$ are surjective maps 
$\Gamma, \Gamma'$ are coloured graphs for ${\sf TCA}_{n+1, n}$, that are not red, then for any $i<n$, we have 
$$([c],[d])\in T_i\Longleftrightarrow ([c],[d])\in T_i^s.$$

If $\Gamma$ is red using the colours for the rainbow signature of $\At$ (without $\rho$) 
and $a':n\to \Gamma$,then for any $b:n\to \Gamma'$ where $\Gamma'$ is not red and any $i<n$, we have   
$$([a'], [b])\in T_i\Longleftrightarrow  ([o(a')], [b])\in T_i^{s}.$$
Extending the notation, for $a:n\to \Gamma$ a graph that is not red in $\At$, set $o(a)=a$.
Then for any $a:n\to \Gamma$, $b:n\to \Gamma'$, where $\Gamma, \Gamma'$ are 
coloured graphs at least  one of which is not red in $\At$ and any $i<n$, we have
$$[a]T_i[b]\Longleftrightarrow [o(a)]T_i^s[o(b)].$$

Now we deal with the last case, when the two graphs involved are red. 
Now assume that $a':n\to \Gamma$ is as above, that is $\Gamma\in {\sf CGR}$ is red, 
$b:n\to \Gamma'$ and $\Gamma'$ is red too, using the colours in the rainbow signature $\At$.

Say that two maps $a:n\to \Gamma$, $b:n\to \Gamma'$, with $\Gamma$ and $\Gamma'\in \sf CGR$ having the same size 
are  $\r$ related if all non red edges and $n-1$ tuples have the same colours (whenever defined), and 
for all every red $\r$, whenever  $i<j<n$, $l\in \omega$,  and $(a(i), a(j))\in \r^l$, then there exists
$k\in \omega$ such that $(b(i), b(j))\in \r^k$.
Let $i<n$. Assume that $([o(a')], [o(b)])\in T_i^s$. Then there exists $c:n\to \Gamma$ that is $\r$ related to $a'$
such that $[c]T_i[b]$.  Conversely, if $[c]T_i[b]$, then $[o(c)]T_i[o(b)].$

Hence, by complete additivity of cylindrifiers,  the map $\Theta: \At({\sf TCA}_{n+1, n})\to \Cm\At$ defined via
\[
 \Theta(\{[a]\})=
  \begin{cases}
    \{ [a']: \text { $a'$   copy of $a$}\}  \text { if $a$ is red, } \\
        \{[a]\} \text { otherwise. } \\
   
  \end{cases}
\]
\\
induces an embedding from ${\sf TCA}_{n+1, n}$ to $\Cm\At$, which we denote also by 
$\Theta$.

We first check preservation of diagonal elements. 
If $a'$ is a copy of $a$, $i, j<n$, and  $a(i)=a(j)$, then $a'(i)=a'(j)$.

We next  check cylindrifiers. We show that for all $i<n$ and $[a]\in \At(\TCA_{n+1, n})$ we have: 
$$\Theta({\sf c}_i[a])=\bigcup \{\Theta([b]):[b]\in \At\TCA_{n+1,n}, [b]\leq {\sf c}_i[a]\}= {\sf c}_i\Theta ([a]).$$
Let $i<n$. If $[b]\in \At\TCA_{n+1,n}$,  $[b]\leq {\sf c}_i[a]$, and $b':n\to \Gamma$, $\Gamma\in \sf CGR$, is a copy of $b$, 
then there exists  $a':n\to \Gamma'$, $\Gamma'\in \sf CGR$, a copy of  $a$ such that
$b'\upharpoonright  n\setminus \{i\}=a'\upharpoonright n\setminus \{i\}$. Thus $\Theta([b])\leq {\sf c}_i\Theta([a])$.

Conversely, if $d:n\to \Gamma$, $\Gamma\in \sf CGR$ and $[d]\in {\sf c}_i\Theta([a])$, then there exist $a'$ a copy of $a$ such that 
$d\upharpoonright n\setminus \{i\}=a'\upharpoonright n\setminus \{i\}$. 
Hence $o(d)\upharpoonright n\setminus  \{i\}=a\upharpoonright n\setminus \{i\}$, and so $[d]\in \Theta({\sf c}_i[a]),$
and we are done.

But now we can show that \pa\ can win the game $F^{n+3}$ on $\At({\sf TCA}_{n+1,n})$ in only $n+2$ rounds 
as follows. 
Viewed as an \ef\ forth game  pebble game, with finitely many rounds and pairs of pebbles, 
played on the two complete irreflexive graphs $n+1$ and $n$, in each round $0,1\ldots n$, \pa\ places a  new pebble  on  an element of $n+1$.
The edge relation in $n$ is irreflexive so to avoid losing
\pe\ must respond by placing the other  pebble of the pair on an unused element of $n$.
After $n$ rounds there will be no such element,
and she loses in the next round.
Hence \pa\ can win the graph game on $\At({\sf TCA}_{n+1,n})$ in $n+2$ rounds using  $n+3$ nodes.

In the game $F^{n+3}$ \pa\ forces a win on a red clique using his excess of greens by bombarding \pe\
with $\alpha$ cones having the same base ($1\leq \alpha\leq n+2)$.

In his zeroth move, \pa\ plays a graph $\Gamma$ with
nodes $0, 1,\ldots, n-1$ and such that $\Gamma(i, j) = \w_0 (i < j <
n-1), \Gamma(i, n-1) = \g_i ( i = 1,\ldots, n-2), \Gamma(0, n-1) =
\g^0_0$, and $ \Gamma(0, 1,\ldots, n-2) = \y_{n+2}$. This is a $0$-cone
with base $\{0,\ldots , n-2\}$. In the following moves, \pa\
repeatedly chooses the face $(0, 1,\ldots, n-2)$ and demands a node
$\alpha$ with $\Phi(i,\alpha) = \g_i$, $(i=1,\ldots n-2)$ and $\Phi(0, \alpha) = \g^\alpha_0$,
in the graph notation -- i.e., an $\alpha$-cone, without loss $n-1<\alpha\leq  n+1$,  on the same base.
\pe\ among other things, has to colour all the edges
connecting new nodes $\alpha, \beta$ created by \pa\ as apexes of cones based on the face $(0,1,\ldots, n-2)$, that is $\alpha,
\beta\geq n-2$. 
By the rules of the game
the only permissible colours would be red. Using this, \pa\ can force a
win in $n+2$ rounds, using $n+3$ nodes  without needing to re-use them, 
thus forcing \pe\ to deliver an inconsistent triple 
of reds.

Let $\B={\sf TCA}_{n+1, n}$. 
Then $\Rd_{sc}\B$   is
outside $S\Nr_n\Sc_{n+3}$ for if it was in $S\Nr_n\Sc_{n+3}$, then being finite it would be in $S_c\Nr_n\Sc_{n+3}$
because $\Rd_{sc}\B$ is the same as its canonical extension $\D$, say, and $\D\in S_c\Nr_n\Sc_{n+3}$. 
But then by theorem \ref{thm:n}, \pe\ would have won. 

Hence $\Rd_{sc}\mathfrak{Cm}\At\notin S\Nr_n\Sc_{n+3}$,  
because $\Rd_{sc}\B$ is embeddable in it
and $S\Nr_n\Sc_{n+3}$ is a variety; in particular, it is closed
under forming subalgebras. 
It now readily follows that $\Rd_{sc}\Cm\At\notin S\Nr_n{\sf Sc}_{n+3}$.

Finally $\Rd_{df}\A$ is not completely representable, because if it were then $\A$,
generated by elements whose dimension sets $<n$,  as a $\TCA_n$ would be completely representable 
and this induces a representation of its \d\ completion $\Cm\At\A$.

%Moving backwards here $\At$ is obtained by blowing up and blurring the atom structure
%of $\TCA_{n+2, n+1}$.; each red in the latter is split to $\omega$ many copies.
\end{proof}

\subsection{Non-finite axiomatizability}

Now we deal with fine non-finite axiomatizability results. We use the construction
in \cite{HHbook} together with the lifting argument used in \cite{t}
to prove a very strong non-finite axiomatizability results expressed by excluding {\it finite schema}. 
Of course $\sf RTCA_{\alpha}$ {\it cannot} be finitely axiomatizable
for the simple reason that its signature has infinitely many operations. This is the case with ${\sf TCA}_{\alpha}$, too, 
but one cannot  help but `sense'  that ${\sf TCA}_{\alpha}$ is axiomatizated by some {\it finite schema}; and indeed it is. 

The axiomatization is {\it finitary} in  a {\it two sorted} sense, one for
ordinals $<\alpha$ and the other for the first order situation. This can be formulated in such a way that there is a {\it strict finite set of equations}
in the signature of $\sf TCA_{\omega}$ such that the axiomatization of ${\sf TCA}_{\alpha}$, for any ordinal $\alpha\geq \omega$,
consists of {\it all $\alpha$ instances}
of such equations. Such a situation is best formulated in the context of {\it systems of varieties definable by a Monk's schema} 
\cite[Definitions, 5.6.11-5.6.12]{HMT2}. 
This is not the case for $\sf RTCA_{\alpha}$, and each of its approximations $S\Nr_{\alpha}{\sf TCA}_{\alpha+k}$, $k\geq 2$,
as we proceed to show.

For this purpose we show that for any ordinal $\alpha>2$, for any  $r\in \omega$,
and for any $k\geq 1$, there exists
$\B^r\in \SNr_{\alpha}\TCA_{\alpha+k}\sim S\Nr_{\alpha}\TCA_{\alpha+k+1}$ such
that $\Pi_{r/U}\B^r\in \sf TRCA_{\alpha}$, for any non-principal ultrafilter on $\omega$.
We will use quite sophisticated constructions of Hirsch and Hodkinson for relation and cylindric algebras reported 
in \cite{HHbook}.

Assume that $3\leq m\leq n<\omega$. For $r\in \omega$,  let $\C_r=\Ca(H_m^{n+1}(\A(n,r),  \omega))$
as defined in \cite[definition 15.3]{HHbook}.
We denote $\C_r$ by $\C(m,n,r)$.
Then the following hold:
\begin{lemma}\label{2.12}
\begin{enumarab}
\item For any $r\in \omega$ and $3\leq m\leq n<\omega$, we
have $\C(m,n,r)\in \Nr_m{\sf CA}_n,$ $\C(m,n,r)\notin S\Nr_m{\sf CA_{n+1}}$
and $\Pi_{r/U}\C(m,n,r)\in {\sf RCA}_m.$ Furthermore, for any $k\in \omega$, 
$\C(m, m+k, r)\cong \Nr_m\C(m+k, m+k, r).$

\item  If $3\leq m<n$, $k\geq 1$ is finite,  and $r\in \omega$, there exists $x_n\in \C(n,n+k,r)$
such that $\C(m,m+k,r)\cong \Rl_{x_n}\C(n, n+k, r)$ and ${\sf c}_ix_n\cdot {\sf c}_jx_n=x_n$
for all $i,j<m$.

\end{enumarab}
\end{lemma}
\begin{proof}
\begin{enumarab}

\item  Assume that $3\leq m\leq n<\omega$, and let
$$\mathfrak{C}(m,n,r)=\Ca(H_m^{n+1}(\A(n,r),  \omega)),$$
be as defined in \cite[Definition 15.4]{HHbook}.
Here $\A(n,r)$ is a finite Monk-like relation algebra \cite[Definition 15.2]{HHbook}
which has an $n+1$-wide $m$-dimensional hyperbasis $H_m^{n+1}(\A(n,r), \omega)$
consisting of all $n+1$-wide $m$-dimensional  wide $\omega$ hypernetworks \cite[Definition 12.21]{HHbook}.
For any $r$ and $3\leq m\leq n<\omega$, we have $\mathfrak{C}(m,n,r)\in \Nr_m{\sf CA}_n$.
Indeed, let $H=H_n^{n+1}(\A(n,r), \omega)$. Then $H$  is an $n+1$-wide $n$ dimensional $\omega$ hyperbasis,
so $\Ca H\in {\sf CA}_n.$ But, using the notation in \cite[Definition 12.21 (5)]{HHbook},
we have  $H_m^{n+1}(\A(n,r),\omega)=H|_m^{n+1}$.
Thus
$$\mathfrak{C}(m,n,r)=\Ca(H_m^{n+1}(\A(n,r), \omega))=\Ca(H|_m^{n+1})\cong \Nr_m\Ca H.$$
The second part is proved in \cite[Corollary 15.10]{HHbook},
and the third in \cite[exercise 2, p. 484]{HHbook}.
We consider  $\mathfrak{C}(m,n,r)$ expanded by $m$ unary operations, namely, each equal to the identity.
\item  Let $3\leq m<n$. Take $$x_n=\{f:{}^{\leq n+k+1}n\to \At\A(n+k, r)\cup \omega:  m\leq j<n\to \exists i<m, f(i,j)=\Id\}.$$
Then $x_n\in C(n,n+k,r)$ and ${\sf c}_ix_n\cdot {\sf c}_jx_n=x_n$ for distinct $i, j<m$.
Furthermore
\[{I_n:\C}(m,m+k,r)\cong \Rl_{x_n}\Rd_m {\C}(n,n+k, r),\]
via the map, defined for $S\subseteq H_m^{m+k+1}(\A(m+k,r), \omega)),$ by
$$I_n(S)=\{f:{}^{\leq n+k+1}n\to \At\A(n+k, r)\cup \omega: f\upharpoonright {}^{\leq m+k+1}m\in S,$$
$$\forall j(m\leq j<n\to  \exists i<m,  f(i,j)=\Id)\}.$$
\end{enumarab}
\end{proof}
An analogous result to the coming theorem is proved 
for cylindric algebras in \cite{logica}, with a precursor in \cite{t}. The argument used in all three proofs is basically a lifting argument 
initiated by Monk \cite[Theorem 3.2.67]{HMT2}.

\begin{theorem}\label{new} Let $\alpha>2$ be an  ordinal. Then for any $r\in \omega$, for any
finite $k\geq 1$, for any $l\geq k+1$ (possibly infinite),
there exist $\B^{r}\in S\Nr_{\alpha}\TCA_{\alpha+k}\sim S\Nr_{\alpha}\TCA_{\alpha+k+1}$ such
$\Pi_{r\in \omega}\B^r\in S\Nr_{\alpha}\TCA_{\alpha+l}$.
In particular, for any such $k$ and $l$, and for $\alpha$ finite, $S\Nr_{\alpha}\TCA_{\alpha+l}$ is not finitely axiomatizable over
$S\Nr_{\alpha}\TCA_{\alpha+k}$, and for infinite $\alpha$,  $S\Nr_{\alpha}\TCA_{\alpha+l}$ is not axiomatizable
by a finite schema over $S\Nr_{\alpha}\TCA_{\alpha+k}$.
\end{theorem}
\begin{proof} 
By lemma \ref{2.12} we can assume that $\alpha$ is infinite. We first show that
$S\Nr_{\alpha}\TCA_{\alpha+k+1}\neq S\Nr_{\alpha}\TCA_{\alpha+k};$ furthermore we construct infinitely many algebras
that witness the strictness
of the inclusion, one for each $r\in \omega$. Their ultraproduct relative to any non-principal ultrafilter on $\omega$
will be representable. We use the algebras $\C(m,n,r)$ in theorem \ref{2.12} in the signature of $\sf TCA_m$, by interpreting 
the interior operator $I_i$ for each $i<m$
as the identity function, so we still have $\C(m, m+k, r)\cong \Nr_m\C(m+k, m+k, r)$ 
for any $k\in \omega$.

Fix $r\in \omega$.
Let $I=\{\Gamma: \Gamma\subseteq \alpha,  |\Gamma|<\omega\}$.
For each $\Gamma\in I$, let $M_{\Gamma}=\{\Delta\in I: \Gamma\subseteq \Delta\}$,
and let $F$ be an ultrafilter on $I$ such that $\forall\Gamma\in I,\; M_{\Gamma}\in F$.
For each $\Gamma\in I$, let $\rho_{\Gamma}$
be a one to one function from $|\Gamma|$ onto $\Gamma.$

Let ${\C}_{\Gamma}^r$ be an algebra similar to $\TCA_{\alpha}$ such that
\[\Rd^{\rho_\Gamma}{\C}_{\Gamma}^r={\C}(|\Gamma|, |\Gamma|+k,r).\]
Let
\[\B^r=\Pi_{\Gamma/F\in I}\C_{\Gamma}^r.\]
Then it can be proved like in \cite{t} that 
\begin{enumerate}
\item\label{en:1} $\B^r\in S\Nr_\alpha\TCA_{\alpha+k},$ 
\item\label{en:2} $\Rd_{ca}\B^r\not\in S\Nr_\alpha\CA_{\alpha+k+1},$
\item\label{en:3} $\Pi_{r/U}\B^r\in \sf RTCA_{\alpha}.$
\end{enumerate}

For the first part, for each $\Gamma\in I$ we know that $\C(|\Gamma|+k, |\Gamma|+k, r) \in\TCA_{|\Gamma|+k}$ and
$\Nr_{|\Gamma|}\C(|\Gamma|+k, |\Gamma|+k, r)\cong\C(|\Gamma|, |\Gamma|+k, r)$.
Let $\sigma_{\Gamma}$ be a one to one function
 $(|\Gamma|+k)\rightarrow(\alpha+k)$ such that $\rho_{\Gamma}\subseteq \sigma_{\Gamma}$
and $\sigma_{\Gamma}(|\Gamma|+i)=\alpha+i$ for every $i<k$. Let $\A_{\Gamma}$ be an algebra similar to a
$\CA_{\alpha+k}$ such that
$\Rd^{\sigma_\Gamma}\A_{\Gamma}=\C(|\Gamma|+k, |\Gamma|+k, r)$.
Then, clearly
 $\Pi_{\Gamma/F}\A_{\Gamma}\in \TCA_{\alpha+k}$.

We prove that $\B^r\subseteq \Nr_\alpha\Pi_{\Gamma/F}\A_\Gamma$.  Recall that $\B^r=\Pi_{\Gamma/F}\C^r_\Gamma$ and note
that $\C^r_{\Gamma}\subseteq A_{\Gamma}$
(the universe of $\C^r_\Gamma$ is $C(|\Gamma|, |\Gamma|+k, r)$, the universe of $\A_\Gamma$ is $C(|\Gamma|+k, |\Gamma|+k, r)$).
So, for each $\Gamma\in I$,
\begin{align*}
\Rd^{\rho_{\Gamma}}\C_{\Gamma}^r&=\C((|\Gamma|, |\Gamma|+k, r)\\
&\cong\Nr_{|\Gamma|}\C(|\Gamma|+k, |\Gamma|+k, r)\\
&=\Nr_{|\Gamma|}\Rd^{\sigma_{\Gamma}}\A_{\Gamma}\\
&=\Rd^{\sigma_\Gamma}\Nr_\Gamma\A_\Gamma\\
&=\Rd^{\rho_\Gamma}\Nr_\Gamma\A_\Gamma
\end{align*}
%$\Rd^{\rho_\Gamma}\A_\Gamma \in \K_{|\Gamma|}$, for each $\Gamma\in I$  then $\Pi_{\Gamma/F}\A_\Gamma\in \K_\alpha$.
Thus (using a standard Los argument) we have:
$\Pi_{\Gamma/F}\C^r_\Gamma\cong\Pi_{\Gamma/F}\Nr_\Gamma\A_\Gamma=\Nr_\alpha\Pi_{\Gamma/F}\A_\Gamma$,
proving \eqref{en:1}.

Now we prove \eqref{en:2}.
For this assume, seeking a contradiction, that $\B^r\in S\Nr_{\alpha}\TCA_{\alpha+k+1}$,
then $\Rd_{ca}\B^r\subseteq \Nr_{\alpha}\C$, where  $\C\in \CA_{\alpha+k+1}$.
Let $3\leq m<\omega$ and  $\lambda:m+k+1\rightarrow \alpha +k+1$ be the function defined by $\lambda(i)=i$ for $i<m$
and $\lambda(m+i)=\alpha+i$ for $i<k+1$.
Then $\Rd^\lambda\C\in \CA_{m+k+1}$ and $\Rd_m\Rd_{ca}\B^r\subseteq \Nr_m\Rd^\lambda\C$.

For each $\Gamma\in I$,\/  let $I_{|\Gamma|}$ be an isomorphism
\[{\C}(m,m+k,r)\cong \Rl_{x_{|\Gamma|}}\Rd_m {\C}(|\Gamma|, |\Gamma+k|,r).\]
Exists by item (2) of lemma \ref{2.12}.
Let $x=(x_{|\Gamma|}:\Gamma)/F$ and let $\iota( b)=(I_{|\Gamma|}b: \Gamma)/F$ for  $b\in \C(m,m+k,r)$.
Then $\iota$ is an isomorphism from $\C(m, m+k,r)$ into $\Rl_x\Rd_m\B^r$.
Then by \cite[theorem~2.6.38]{HMT1} we have $\Rl_x\Rd_{m}\Rd_{ca}\B^r\in S\Nr_m\CA_{m+k+1}$.
It follows that  $\C (m,m+k,r)\in S\Nr_{m}\CA_{m+k+1}$ which is a contradiction and we are done.

Now we prove \eqref{en:3} putting the superscript $r$ to use.
Recall that $\B^r=\Pi_{\Gamma/F}\C^r_\Gamma$, where $\C^r_\Gamma$ has the type of $\TCA_{\alpha}$
and $\Rd^{\rho_\Gamma}\C^r_\Gamma=\C(|\Gamma|, |\Gamma|+k, r)$.
We know 
from item (1) of lemma \ref{2.12} that $\Pi_{r/U}\Rd^{\rho_\Gamma}\C^r_\Gamma=\Pi_{r/U}\C(|\Gamma|, |\Gamma|+k, r) \subseteq \Nr_{|\Gamma|}\A_\Gamma$,
for some $\A_\Gamma\in\TCA_{|\Gamma|+\omega}$.

Let $\lambda_\Gamma:|\Gamma|+k+1\rightarrow\alpha+k+1$
extend $\rho_\Gamma:|\Gamma|\rightarrow \Gamma \; (\subseteq\alpha)$ and satisfy
\[\lambda_\Gamma(|\Gamma|+i)=\alpha+i\]
for $i<k+1$.  Let $k+1\leq l\leq \omega$.
Let $\F_\Gamma$ be a $\TCA_{\alpha+l}$ type algebra such that $\Rd^{\lambda_\Gamma}\F_\Gamma=\Rd_l\A_\Gamma$.
As before, $\Pi_{\Gamma/F}\F_\Gamma\in\TCA_{\alpha+l}$.  And
\begin{align*}
\Pi_{r/U}\B^r&=\Pi_{r/U}\Pi_{\Gamma/F}\C^r_\Gamma\\
&\cong \Pi_{\Gamma/F}\Pi_{r/U}\C^r_\Gamma\\
&\subseteq \Pi_{\Gamma/F}\Nr_{|\Gamma|}\A_\Gamma\\
&=\Pi_{\Gamma/F}\Nr_{|\Gamma|}\Rd^{\lambda_\Gamma}\F_\Gamma\\
&\subseteq\Nr_\alpha\Pi_{\Gamma/F}\F_\Gamma,
\end{align*}
Hence, we get that  $\Pi_{r/U}\B^r\in S\Nr_{\alpha}\TCA_{\alpha+l}$
and we are done.

Now we show that for $k\geq 1$ and $l\geq k+1$, there is no finite set of equations
in the language of $\TCA_{\omega}$ $E$, such that its $\alpha$ instances axiomatize $S\Nr_{\alpha}\TCA_{\alpha+l}$ over
$S\Nr_{\alpha}\TCA_{\alpha+k}$.

By an $\alpha$ instance of an equation in the signature of $\TCA_{\omega}$ is meant the following.
If $\rho:\omega\to \alpha$ is an injection, then $\rho$ extends recursively
to a function $\rho^+$ from $\TCA_{\alpha}$ terms to $\TCA_{\alpha}$ terms.
On variables $\rho^+(v_k)=v_k$, and for compound terms
like ${\sf c}_k\tau$, where $\tau$ is a $\TCA_{\omega}$ term, and $k<\omega$, $\rho^+({\sf c}_k\tau)={\sf c}_{\rho(k)}\rho^+(\tau)$.
For an equation $e$ of the form $\sigma=\tau$ in the language of $\TCA_{\omega}$, $\rho^+(e)$ is
the equation $\rho^+(\tau)=\rho^+(\sigma)$ in the language
of $\TCA_{\alpha}$. This last equation, namely, $\rho^+(e)$ is called an $\alpha$ instance of $e$
obtained by applying the injection $\rho$.

Let $k\geq 1$ and $l\geq k+1$. Assume for contradiction that $S\Nr_{\alpha}\TCA_{\alpha+l}$ is axiomatizable
by a finite schema over  $S\Nr_{\alpha}\TCA_{\alpha+k}$.
We can assume without loss  that there is only one equation in the signature of $\TCA_{\omega}$,
such that all its $\alpha$ instances,  axiomatize  $S\Nr_{\alpha}\TCA_{\alpha+l}$ over
$S\Nr_{\alpha}\TCA_{\alpha+k}$.
So let $\sigma$ be such an  equation and let $E$ be its $\alpha$ instances; so that
 for any $\A\in S\Nr_{\alpha}\TCA_{\alpha+k}$ we have $\A\in \bold S\Nr_{\alpha}\TCA_{\alpha+l}$ iff
$\A\models E$.  Then for all $r\in \omega$, there is an instance of $\sigma$,  $\sigma_r$ say,
such that $\B^r$ does not model $\sigma_r$.
$\sigma_r$ is obtained from $\sigma$ by some injective map $\mu_r:\omega\to \alpha$.

For $r\in \omega,$ let $v_r\in {}^{\alpha}\alpha$,
be an injection such that $\mu_r(i)=v_r(i)$ for each index $i$ appearing in $(\sigma_r)$, 
and let $\A_{r}= \Rd^{v_r}\B^r$.
Now $\Pi_{r/U} \A_{r}\models \sigma$. But then
$$\{r\in \omega: \A_{r}\models \sigma\}=\{r\in \omega: \B^r\models \sigma_r\}\in U,$$
contradicting that
$\B^r$ does not model $\sigma_r$ for all $r\in \omega.$
\end{proof}

\subsection{Decidability issues}

\begin{theorem} It is undecidable to tell whether a finite $\sf TCA_n$ $n>2$ is representable or not.
\end{theorem}
\begin{proof} Let $n>2$. Assume that there is a desicion precedure to tell. Let $\A$ be a given simple finite $\CA_n$. 
Consider $\A$ as a $\TCA_n$
expanded with interior operators defined as the identity map, call it $\A^+$. Then we can decide whether $\A^+$ is representable or not as a 
$\sf TCA_n$. But $\A^+$ is representable iff $\A=\Rd_{ca}\A^+$ is representable, 
hence we get a decision procedure that tells whether $\A$
as a $\CA_n$ is representable or not. This contradicts \cite{AU}.
\end{proof}

The following corollary is a consequence of the above lemma and of the undecidability result that we have just proved, witness
\cite[corollary 18.16, theorem 18.27]{HHbook} for similar results for relation algebras.

For a class $\sf K$ of algebras, the class ${\sf \K}\cap \sf Fin$ denotes the class of finite members of $\K$.

\begin{corollary}\label{undecidability} Let $2<n<\omega$. Then the following hold 
\begin{enumarab}

\item The set of isomorphism types of finite algebras in $\TCA_n$ with
only infinite representations is not recursively enumerable.

\item The equational theory of $\sf RTCA_n$ is undecidable.

\item The equational theory of $\sf RTCA_n\cap \sf Fin$ is undecidable.

\item The variety $\sf RTCA_n$  is not finitely axiomatizable even
in $m$th order logic.

\end{enumarab}
\end{corollary}


\begin{thebibliography}{99}

\bibitem{Andreka} H. Andr\'eka, {\it Complexity of equations valid in algebras of relations.} Ann Pure and App Logic {\bf 89}(1997) 
p. 149-209.


\bibitem{1} H. Andr\'eka, M. Ferenczi and  I. N\'emeti (Editors), {\bf Cylindric-like Algebras and Algebraic Logic},
Bolyai Society Mathematical Studies and Springer-Verlag, {\bf 22} (2012).

\bibitem{AGN}H. Andr\'eka, T. Gregely H.  and I. N\'emeti, 
{\it On universal algebraic constructions of logics.} 
Studia Logica , {\bf 36} (1977) p.9-47.

\bibitem{AMN}H. Andr\'eka, J.D. Monk., I. N\'emeti,I. (editors) {\it Algebraic Logic,}
North- Holland, Amsterdam, 1991.


\bibitem{AN75} Andr\'eka,H., N\'emeti,I., 
{\it A simple purely algebraic proof of the completeness of some first order logics.}
Algebra Universalis, {\bf 5} (1975) p.8-15.


\bibitem{AN80} Andr\'eka,H.,Nemeti,I., {\it On Systems of varieties definable by schemes of equations.} 
Algebra Universalis, {\bf 11}(1980) p. 105-116.

\bibitem{ANS} Andr\'eka, H.,  N\'emeti, I., Sain, I {\it Algebraic Logic} (2000).
In Handbook of Philosophical Logic, Editors Gabbay et all.

\bibitem{ANT} H. Andr\'eka, I. N\'emeti and T. Sayed Ahmed, 
{\it Omitting types for finite variable fragments and complete representations.}
Journal of Symbolic Logic {\bf 73} (2008) p. 65-89 

\bibitem{BP} Blok,W.J., and Pigozzi,D. {\it Algebraizable logics}. 
Memoirs of American  Mathematical Society, {\bf 77}(396), 1989.

\bibitem{Comer} Comer S.D. {\it Classes without the amalgamation property} 
Pacific journal of Mathematics, {\bf 28}(2)(1969), p.309-318. 



\bibitem{Chang} Chang {\it Modal model theory.} Proceedings of the Cambridge Summer School in mathematical logic, Lecture Notes 
{\bf 337} (1974) p. 599-617.

\bibitem{D}A. Daigneault, {\it Freedom in polyadic algebras and two theorems
of Beth and Craig}. Michigan Math. J. {\bf 11}(1963), p. 129-135.

\bibitem{DM} A. Daigneault and J.D. Monk,
{\it Representation Theory for Polyadic algebras}.
Fund. Math. {\bf 52}(1963), p.151-176.


\bibitem{g4} G.Georgescu {\it A representation theorem for tense polyadic algebras.} Mathematiuca, Tome 21 {\bf 44} (2) 
(1979) p.131-138.

\bibitem{g3} G. Georgescu {\it Modal polyadic algebras.} Bull. Math Soc. Sci Math R,S Romaina (1979) {\bf 23} p.49-64


\bibitem{g} G. Georgescu {\it Algebraic analysis of topological logic.} Mathematical Logic Quarterly (28) p.447-454 (1982) 
{\bf 52}(5)(2006) p.44-49.

\bibitem{g5} G. Georgescu, {\it A representation theorem for polyadic Heyting algebras.}
Algebra Universalis, {\bf 14} (1982) , p.197-209.

\bibitem{g2} G. Georgescu {\it Chang's modal operators in Algebraic Logic.} Studia Logica {\bf 42}(1), (1983) p.43-48


\bibitem{Halmos} P. Halmos, {\it Algebraic Logic.}
Chelsea Publishing Co., New York, (1962.)


\bibitem{Henkin} L. Henkin, {\it An extension of the Craig-Lyndon interpolation theorem.} 
Journal of Symbolic Logic {\bf 28}(3) (1963), p.201-216

\bibitem{HMT1}L. Henkin, J.D. Monk  and A.Tarski,  {\it Cylindric Algebras Part I}.
North Holland, 1971.

\bibitem{HMT2} L. Henkin, J.D. Monk and A. Tarski, {\it Cylindric Algebras Part II}.
North Holland, 1985.


\bibitem{cat}H. Herrlich, G. Strecker, {\it Category theory.} Allyn and Bacon, Inc, Boston (1973)

\bibitem{r} R. Hirsch, {\it Relation algebra reducts of cylindric algebras and complete representations},
Journal of Symbolic Logic, {\bf 72}(2) (2007), p.673-703.



\bibitem{HH} Hirsch R., Hodkinson, I. {\it Complete representations in algebraic logic}
Journal of Symbolic Logic, 62, 3 (1997) 816-847 


\bibitem{HHbook} R. Hirsch and I. Hodkinson, {\it Relation algebras by games.}
Studies in Logic and the Foundations of Mathematics, volume {\bf 147} (2002)

\bibitem {HHbook2} R. Hirsch and I. Hodkinson, {\it Completions and complete representations in algebraic logic.} In \cite{1}

\bibitem{t} R. Hirsch and T. Sayed Ahmed, {\it The neat embedding problem for algebras other than cylindric algebras
and for infinite dimensions.} Journal of Symbolic Logic (in press).


\bibitem{Hodkinson} I. Hodkinson, {\it Atom structures of relation and cylindric algebras}. Annals of pure and applied logic,
{\bf 89}(1997), p.117-148.

\bibitem {AU} I. Hodkinson, \emph{A construction of cylindric and polyadic algebras from atomic relation algebras.}
Algebra Universalis, {\bf 68} (2012), p. 257-285.

\bibitem{k} A.S. Kechris, \emph{Classical Descriptive Set Theory}, Springer Verlag, New York (1995).


\bibitem{k} M. Khaled and T. Sayed Ahmed, {\it Classes of Algebras not closed under completions.}
Bulletin Section of Logic {\bf 38}(2009) p. 29-43.

\bibitem{Mad} J. Madar\'asz
{\it Interpolation and Amalgamation;  Pushing the Limits. Part I}
Studia Logica, {\bf 61}, (1998) p. 316-345.

%\bibitem{Mad98} Madar\'asz, J. {\it interpolation and Amalgamation; 
%Pushing the Limits. Part II} Studia Logica.

\bibitem{AUU} J. Mad\'arasz and T. Sayed Ahmed,
{\it Amalgamation, interpolation and epimorphisms.}
Algebra Universalis {\bf 56} (2) (2007), p. 179 - 210, 

\bibitem{MStwo} J. Mad\'arasz and T. Sayed Ahmed,
{\it Neat reducts and amalgamation in retrospect, a survey of results and some
methods. Part 2: Results on amalgamation.} Logic Journal of IGPL {\bf 17}, (2009), p.755-802.

\bibitem{MS} J. Mad\'arasz and T. Sayed Ahmed
{\it Amalgamation, interpolation and epimorphisms in algebraic logic}.  In \cite{1}, p.91-104

\bibitem{z} Makowski and Ziegler {\it Topological model theory with an interior operator.} Preprint


\bibitem{Mak} L. Maksimova 
{\it Amalgamation and interpolation in normal modal logics}.
Studia Logica {\bf 50}(1991) p.457-471. 

\bibitem{z2} J. Marowski  and A. Marcia {\it Completeness theorem for modal model theory with Montague Chang semantics.} 
This Zeisachr {\bf  23} (1977) 97-104.



\bibitem{george} G. Metcalfe, F. Montagna , and C. TsiNakis {\it Amalgamation and interpolation in ordered algebras.} 
(2012) pre-print.



\bibitem{P} D. Pigozzi,
{\it Amalgamation, congruence extension, and interpolation properties in algebras.}
Algebra Universalis.
{\bf 1}(1971), p.269-349.



\bibitem{Sagi} G. S\'agi and D. Szir\'aki, \emph{Some Variants of Vaught's Conjecture from the Perspective of Algebraic Logic}, 
Logic Journal of the IGPL, published online January 5, 2012.


\bibitem{ak} S. Awodey and K. Kishida {\it Topology and modality, the topology of first order models.} 
Review of Symbolic Logic {\bf 1} 
(2008) p. 146-166.

\bibitem{S} I. Sain {\it Searching  for a finitizable algebraization of first order logic}. 
8, Logic Journal of IGPL. Oxford University, Press (2000) no 4, p.495--589.

\bibitem{ST} I. Sain and R. Thompson,
{\it Strictly finite schema axiomatization of quasi-polyadic algebras}, in
{\bf Algebraic Logic} H. Andr\'eka, J. D. Monk and I. N\'emeti (Editors),  North Holland (1990) p.539-571.


\bibitem{IGPL2}T. Sayed Ahmed, {\it The class of neat reducts is not elementary.} Logic Journal of $IGPL$, {\bf 9}(2001)p. 593-628.

\bibitem{MLQ} T. Sayed Ahmed, {\it A model-theoretic solution to a problem of Tarski.} Math Logic Quarterly, Vol. 48 (2002), pp. 343-355.

\bibitem{AUamal} T. Sayed Ahmed, {\it On Amalgamation of Reducts of Polyadic Algebras.}
Algebra Universalis  {\bf 51} (2004), p.301-359.

\bibitem{IGPL} T. Sayed Ahmed,
{\it An interpolation theorem for first order logic with infinitary predicates.}
Logic journal of IGPL, {\bf 15}(1) (2007), p.21-32

\bibitem{weak} T. Sayed Ahmed, {\it Weakly representable atom structures that are not strongly representable,
with an application to first order logic,} Mathematical Logic Quarterly, {\bf 54}(3)(2008) p. 294-306.

\bibitem{ex} Sayed Ahmed, T., {\it On finite axiomatizability of expansions of cylindric algebras.} 
Journal of Algebra, Number Theory, Advances and Applications, {\bf 1}(2010),  p.19-40.

\bibitem{super} T. Sayed Ahmed, {\it The class of polyadic algebras has the super amalgamation property}
Mathematical Logic Quarterly {\bf 56}(1)(2010), p.103-112



\bibitem{amal} T. Sayed Ahmed,  {\it Classes of algebras without the amalgamation property.} Logic Journal of IGPL,
{\bf 19} (2011), p. 87-104.

\bibitem{Hung} T. Sayed Ahmed {\it Representability and amalgamation in Heyting polyadic algebras}, Studia Mathematica
Hungarica, {\bf 48}(4)(2011), p. 509-539.

\bibitem{univl} T. Sayed Ahmed, {\it Amalgamation in universal algebraic logic.} Stududia Mathematica Hungarica,
{\bf 49} (1) (2012), p. 26-43.

\bibitem{typeless} T. Sayed Ahmed, {\it Three interpolation theorems for typeless logics.} Logic Journal of $IGPL$ {\bf 20}(6)
(2012), p.1001-1037.

\bibitem{es} T. Sayed Ahmed, {\it Epimorphisms are not surjective even in simple algebras.}
Logic Journal of $IGPL$, {\bf 20}(1) (2012), p. 22-26.


\bibitem{conference} T. Sayed Ahmed, {\it Neat embeddings as adjoint situations.}
Published on Line, Synthese. DOI 10.1007/s11229-013-0344-7.

\bibitem{Sayed} T. Sayed Ahmed {\it Completions, complete representations and omitting types.} In \cite{1}.

\bibitem{Sayedneat} T. Sayed Ahmed, {\it Neat reducts and neat embeddings in cylindric algebras.} In \cite{1}.




\bibitem{logica} T. Sayed Ahmed 
{\it Results on neat embeddings with applications to algebraizable extensions of first order 
logic.} Submitted.

\bibitem{part2} T. Sayed Ahmed {\it An algebraic approach to topological logic and Chang's modal logic 
using cylindric algebras, Part 2\\
{\bf Amalgamation, interpolation and congruence extension properties in topological algebras}.}

\bibitem{part3} T. Sayed Ahmed {\it An algebraic approach to topological logic and Chang's modal logic 
using cylindric algebras, Part 3\\
{\bf Some more algebra; finite dimensional topological cylindric algebras }}


\bibitem{part3} T. Sayed Ahmed {\it An algebraic approach to topological logic and Chang's modal logic 
using cylindric algebras, Part 4\\
{\bf Logical consequences. }}


\bibitem{Shelah} S. Shelah, {\it Classification theory: and the number of non-isomorphic models}
Studies in Logic and the Foundations of Mathematics.  (1990).


\bibitem{Simon} Simon {\it Non-representable algebras of relations.} Ph.D dissertation, 
Mathematical institute, of the Hungarian Academy of Sciences (1997).



\bibitem{s2} Sgro {\it The interior operator logic and product topologies.} Trans. Amer. Math. Soc. {\bf 258}(1980) p. 99-112.



\bibitem{s} Sgro {\it Completeness theorems for topological models.} Annals of Mathematical Logic (1977) p.173-193.

\bibitem{tm} Tarski and Mckinsey {\it The Algebra of topology} Annals of mathematics {\bf 45}(1944) p. 141-191 

\end{thebibliography}
\end{document}